\documentclass[12pt]{amsart}
\usepackage{amsmath, amsthm, amscd, amssymb, amsfonts, latexsym, mathtools}
\usepackage{fullpage}
\usepackage{bm,mathdots}
\usepackage{dsfont}
\usepackage[all]{xy}
\usepackage{tikz}
\usepackage{float}
\usepackage{fancybox}
\usepackage{bbm}
\usepackage{ytableau}
\numberwithin{equation}{section}
\usepackage[colorlinks=true, pdfstartview=FitV,linkcolor=blue,citecolor=blue,urlcolor=blue]{hyperref}

\newenvironment{red}{\relax\color{red}}{\relax}
\newenvironment{blue}{\relax\color{blue}}{\hspace*{.5ex}\relax}
\newenvironment{jaune}{\relax\color{magenta}}{\relax}
\newcommand{\ber}{\begin{red}}
\newcommand{\er}{\end{red}}
\newcommand{\beb}{\begin{blue}}
\newcommand{\eb}{\end{blue}}
\newcommand{\bej}{\begin{jaune}}
\newcommand{\ej}{\end{jaune}}

\newcommand{\kommentar}[1]{}

\newcommand{\F}{\mathbb F}

\newcommand{\Z}{\mathbb Z}
\newcommand{\Q}{\mathbb Q}

\DeclareMathOperator{\tr}{tr}

\renewcommand{\pmod}[1]{\,(\mathrm{mod}\,#1)}

\newtheorem{lem}{Lemma}[section]
\newtheorem{prop}[lem]{Proposition}
\newtheorem{thm}[lem]{Theorem}

\newtheorem{cor}[lem]{Corollary}

\theoremstyle{definition}

\begin{document}

\title{Murmurations of quadratic and cubic characters \\ over function fields}

\author{Matilde Lalín}
\address{D\'epartement de math\'ematiques et de statistique, Universit\'e de Montr\'eal, Montreal, QC H3C 3J7, Canada}\email{matilde.lalin@umontreal.ca}

\author[K.-H. Lee]{Kyu-Hwan Lee}
\address{Department of
Mathematics, University of Connecticut, Storrs, CT 06269, U.S.A. \hfill \break \indent Korea Institute for Advanced Study, Seoul 02455, Republic of Korea}
\email{khlee@math.uconn.edu}

\author{Thomas Oliver}
\address{University of Westminster, London, U.K.}
\email{T.Oliver@westminster.ac.uk}

\author{Alexey Pozdnyakov}
\address{Department of Mathematics, Princeton University, Princeton, NJ 08544-1000, U.S.A.}
\email{ap5763@princeton.edu}

\begin{abstract}
We compute the murmuration density for the family of quadratic characters over function fields. 
We show that the expectation of this density coincides with a correction term arising in the traces of high powers of the Frobenius class in this family, for which we determine large-genus asymptotics. 
This correction term admits a description in terms of the stable homology of the moduli space of hyperelliptic curves of genus $g$ with one marked Weierstrass point. 
We further identify this correction as a lower-order term in the one-level density and we obtain a non-vanishing result at the central point for quadratic $L$-functions over function fields, via an alternative approach to using the ratios conjecture.
We also compute the murmuration density for the thin family of primitive cubic characters in the Kummer setting.
\end{abstract}

\maketitle

\tableofcontents

\section{Introduction}

\subsection{Motivation}
The murmuration phenomenon was first observed empirically in the family of elliptic curves over $\mathbb{Q}$ by He, Lee, Oliver and Pozdnyakov \cite{HLOP} and then established for the family of modular forms of fixed weight and root number by Zubrilina \cite{Zubrilina}. 
Furthermore, as clarified by Sarnak \cite{sarnak2023murmurations}, murmurations provide a refinement of the one-level density. 

The purpose of this article is to investigate the function field analogue of murmurations. 
We begin by computing the murmuration density for the family of quadratic characters over function fields, already known in the number field setting under GRH due to Lee, Oliver and Pozdnyakov \cite{LOP}. 
We observe that the murmuration density is closely connected to a correction term in the traces of high powers of the Frobenius class in this family, which was already established by Rudnick \cite{rudnick2008traces}. 
In particular, this correction term is simply the expectation of the murmuration density. 
As a corollary, it also appears as a lower-order term in the one-level density for this family in the conductor limit.

We then investigate this correction term from a purely cohomological perspective. 
Using recent developments in homological stability for braid groups by Bergstr\"om, Diaconu, Petersen, and Westerland \cite{bergström2024hyperellipticcurvesscanningmap} and by Miller, Patzt, Petersen, and Randal-Williams \cite{miller2025uniformtwistedhomologicalstability}, we prove a formula for the correction term as a function of the stable homology of $\mathcal{H}_g^{1,0}$, the moduli space of hyperelliptic curves of genus $g$ with one marked Weierstrass point. 
The resulting expression can then be computed using the main theorems of \cite{bergström2024hyperellipticcurvesscanningmap}. 

The methods used to compute this correction term also allow us to obtain asymptotics for traces of arbitrarily large powers of the Frobenius class in the limit $g \to \infty$, provided that $q$ is sufficiently large relative to the power. 
As a corollary, we can obtain the one-level density for test functions with Fourier transform supported in an arbitrary interval, provided that $q$ is sufficiently large relative to the interval. 
Using a standard argument, this also implies that 100\% of the $L$-functions of quadratic characters of $\mathbb{F}_q[T]$ do not vanish at the central point in the double limit $\lim_{q \to \infty} \liminf_{g \to \infty}$. 
It was already observed in \cite{wang2024noteszetaratiostabilization} that one can obtain such results from homological stability by proving a form of the ratios conjecture, although our proof is a more direct computation of the traces. 
Similar results follow from \cite{wang2024noteszetaratiostabilization} by applying Cauchy's integral formula to $L'/L(s,\chi_D)$.  

We conclude by computing the murmuration density for the thin Kummer family of primitive cubic characters. 
The situation here differs from that of quadratic characters in several important ways. 
First, this family is non-self-dual. 
Recalling that the murmuration density as defined in \cite{Zubrilina} and \cite{LOP} includes the root number, we must employ different techniques to compute the density. 
Second, this family has unitary symmetry type, so the one-level density no longer has a sharp transition where we expect murmurations. 
Nevertheless, as in the case of the family of all Dirichlet characters in the number field setting, the murmuration density still defines a non-trivial function of the prime $P$ in the large conductor limit. 
Third, this example shows that the murmuration density varies with $P$, in contrast with the quadratic case.

\subsection{Results}
We begin with the case of quadratic characters. 
Let $q$ denote an odd prime power. For a monic polynomial $D \in \mathbb{F}_q[T]$ of positive degree, which is not a perfect square, we define the quadratic character $\chi_D$ on $\mathbb{F}_q[T]$ by 
\begin{equation*}
    \chi_D(f) =\left(\frac{D}{f}\right).
\end{equation*}
We define the associated $L$-function via the Euler product
\begin{equation*}
    L(s, \chi_D) := \prod_{P}(1-\chi_D(P)|P|^{-s})^{-1}, \quad \Re(s) > 1,
\end{equation*}
with product over all monic irreducible (prime) polynomials and $|P| = q^{\deg P}$. Equivalently, this is the $L$-function of the Galois representation attached to the first cohomology of the hyperelliptic curve $C_D: y^2= D(x)$. Using the Weil conjectures for curves over finite fields, we can define a conjugacy class of $2g \times 2g$ unitary symplectic matrices $\Theta_D$ by
\begin{equation*}
    L(s,\chi_D) = \det(I-q^{1/2-s} \Theta_D).
\end{equation*}
We call this the Frobenius conjugacy class for the curve $C_D$.

Let $\mathcal{M}_n$ denote the set of monic polynomials of degree $n$. Similarly, let  $\mathcal{H}_{n}$ and $\mathcal{P}_n$ denote the sets of monic square-free polynomials of degree $n$ and the set of monic irreducible polynomials of degree $n$, respectively. We put the uniform probability measure on $\mathcal{H}_{2g+1}$ and denote by $\langle F \rangle$ the expectation of $F$ over $\mathcal{H}_{2g+1}$ for any function $F$ on $\mathcal{H}_{2g+1}$. We define a murmuration density $M_{q,g} : \mathcal{P}_{2g} \to \mathbb{R}$ for the family of quadratic characters by
\begin{equation}\label{eq.md}
    M_{q,g}(P) := \frac{1}{\# \mathcal{H}_{2g+1}} \sum_{D \in \mathcal{H}_{2g+1}} \chi_D(P)\sqrt{|P|} = \langle \chi_D(P)\sqrt{|P|} \rangle,
\end{equation}
in which the factor of $\sqrt{|P|}$ is comparable to that found in the number field case \cite{LOP}.
The murmuration density is given by the following theorem.
\begin{thm}\label{thm:quadratic_murmuration}
    For all $P \in \mathcal{P}_{2g}$ we have
    \begin{equation*}
        M_{q,g}(P) = \frac{1}{q-1}.
    \end{equation*}
\end{thm}
Via the explicit formula, the average value of the murmuration density over $\mathcal{P}_{2g}$ also appears as a deviation from the random matrix theory prediction in $\langle \mathrm{Tr} \Theta_D^n \rangle$ at $n = 2g$, as observed in \cite[Theorem 1]{rudnick2008traces}. 
We clarify the relationship between the murmuration density and this correction term further in Section~\ref{sec:discussion}. Using the aforementioned developments in homological stability, we also study this deviation from a purely cohomological perspective. 
This study yields the following strengthening of \cite[Theorem 1]{rudnick2008traces}, although with an additional constraint on $q$. 
Through the course of the proof, which rests almost entirely on \cite{bergström2024hyperellipticcurvesscanningmap} and \cite[Proposition 1.5]{miller2025uniformtwistedhomologicalstability}, we will see a cohomological expression for this correction that can be evaluated exactly. 
\begin{thm}\label{thm:cohomological_formula}
    Suppose that $2 \nmid q$ and $q>2^{\max\{4,3/A+2\}}$. 
    Then as $g \to \infty$ we have 
    \begin{equation*}
        \langle \operatorname{Tr} \Theta_D^n \rangle = \int_{\mathrm{USp}_{2g}(\mathbb{C}) } \operatorname{Tr} U^n dU - \frac{\mathbbm{1}_{n=2g}}{q-1} + O\!\left(\mathbbm{1}_{2\mid n}\left(2^{n/2}q^{-n/8}+g2^{6g+n}q^{-\lfloor2Ag+A+B\rfloor/2}\right)\right),
    \end{equation*}
    the constants $A$ and $B$ are from the uniform stable homology range\footnote{cf. Theorem~\ref{thm:stable_range}. Concretely, we can take $A=\frac{1}{34}$ and $B=-\frac{35}{34}$ 
    \cite[Proposition 1.5]{miller2025uniformtwistedhomologicalstability}.}, and $\mathbbm{1}_S$ denotes the indicator function of the subset of positive integers that satisfy the condition $S$. 
\end{thm}
One sees that the murmuration density appears in the main term when $n=2g$, which is related to the fact that it is constant on $\mathcal{P}_{2g}$. 
The $2^{n/2}q^{-n/8}$ error term hides the well-understood arithmetic correction corresponding to even prime powers. 
Note that we have not made an effort to optimize the constants in the error term or the lower bound on $q$. 
We also note that a similar result follows from the work of Wang \cite{wang2024noteszetaratiostabilization} and refer to Section~\ref{sec:discussion} for a further discussion. 

Using Theorem~\ref{thm:cohomological_formula}, we recover the one-level density in the following form. For any even Schwartz function $f \in \mathcal{S}(\mathbb{R})$ and any $N \geq 1$ set
\begin{equation*}
    F(\theta) := \sum_{k\in \Z} f\left(N \left(\frac{\theta}{2\pi}-k\right)\right),
\end{equation*}
which has period $2\pi$ and is localized in an interval of size $\approx 1/N$ in $\mathbb{R}/2\pi \mathbb{Z}$. For a unitary $N \times N$ matrix $U$ with eigenvalues $e^{i\theta_j}$ define
\begin{equation}\label{eq:def_ZfU}
    Z_f(U) := \sum_{j=1}^{N} F(\theta_j).
\end{equation}

\begin{cor}\label{cor:one_level_density}
    Suppose that $2 \nmid q$ and $q>2^{\max\{4,3/A+2\}}$.
    Let $f \in \mathcal{S}(\mathbb{R})$ be even with Fourier transform $\widehat f$ supported in $(-v,v)$. 
    Then there exists constants $C_j(q)$ such that for any $J \geq 1$ 
    \begin{align*}
        \langle Z_f(\Theta_D) \rangle &= \int_{\mathrm{USp_{2g}(\mathbb{C})}} Z_f(U)dU - \frac{\widehat{f}(1)}{g(q-1)} + \sum_{j= 0}^{J-1}C_j(q) \widehat{f}^{(2j)}(0)g^{-2j-1} \\
        &+O_{J,f}\left(g^{-2J}\right) + O_f\left(vg2^{(6+2v)g}q^{-\lfloor 2Ag+A+B\rfloor /2}\right)
    \end{align*}
     as $g \to \infty$.
\end{cor}

The constants $C_j(q)$ can be expressed as sums over primes as in \cite{rudnick2008traces}. 
We highlight that the correction term appears as a lower-order term in the one-level density. 
Here, the murmuration density is least visible, as we have now averaged over primes of degree $2g$, and also averaged over degrees asymptotic to $2g$. 
This also implies a full density nonvanishing result in the double limit where $g \to \infty$ first and then $q \to \infty$.
\begin{cor}\label{cor:non_vanishing}
    Let $v > 0$ and suppose that $2 \nmid q$ and
    $q>\max\{2^{\frac{3}{A}+2},2^\frac{6+2v}{A}\}$. 
    Then
    \begin{equation*}
        \liminf_{g \to \infty}\frac{\#\{D \in \mathcal{H}_{2g+1} : L(1/2,\chi_D) \neq 0\}}{\#\mathcal{H}_{2g+1}} \geq 1 - \frac{1}{4v^2}.
    \end{equation*}
    In particular,
    \begin{equation*}
        \lim_{q \to \infty} 
        \liminf_{g \to \infty} \frac{\#\{D \in \mathcal{H}_{2g+1} : L(1/2,\chi_D) \neq 0\}}{\#\mathcal{H}_{2g+1}} = 1.
    \end{equation*}
\end{cor}

We recall that there are several related non-vanishing results over function fields in the literature.
Bui and Florea \cite{Bui-Florea} computed the one-level density of quadratic Dirichlet $L$-functions over function fields and proved that $94\%$ of the members of this family do not vanish at the central point. David, Florea, and Lal\'in \cite{DFL} computed the first moment to prove that infinitely many central values of $L$-functions associated to cubic characters over $\mathbb{F}_q[T]$ do not vanish, and later used mollified moments to prove a positive proportion of non-vanishing when $q \equiv 2 \pmod 3$ \cite{DFL2}. This is the non-Kummer setting, in which the base field does not contain the cubic roots of unity.

Li \cite{LiHyperellipticVanishing} showed, by algebro-geometric methods, that central vanishing can nevertheless occur in families of quadratic characters over function fields. This was extended to higher-order characters by Donepudi and Li \cite{DonepudiLi}. Ellenberg, Li, and Shusterman \cite{EllenbergLiShusterman} proved related non-vanishing results for Dirichlet $L$-functions associated to characters of order $\ell$, giving bounds for the proportion of vanishing central values that tend to zero as $q\to\infty$ in the ranges they consider. Later, David, Florea, and Lal\'in \cite{DFL3} computed the one-level density to prove a positive proportion of non-vanishing for families of $\ell$-th order characters.

Our argument is different in nature. It proceeds through the one-level density obtained from homological stability and gives a full-density non-vanishing statement in the double limit.

The remainder of our results explore the murmuration density for a family of primitive cubic characters in the Kummer case, following \cite{DFL} for notation.  
Let $q\equiv 1\pmod{3}$ and fix a prime $P$ of degree $n$. We consider the family 
\begin{equation*}
    \mathcal{F}_3(d):=\left\{\chi_c=\left(\frac{\cdot}{c}\right)_3\, :\, c\in \mathcal{H}_d\right\}.
\end{equation*}
When $d = 2n$, we define the non-trivial murmuration density $\widetilde{M}_{q,d} : \mathcal{P}_{n} \to \mathbb{C}$ by
\begin{equation*}
    \widetilde{M}_{q,d}(P) := \frac{1}{\# \mathcal{F}_3(d)} \sum_{\chi_c  \in \mathcal{F}_3(d)} \omega(\chi_c) \overline{\chi_c}(P)\sqrt{|P|} = \langle  \omega(\chi_c) \overline{\chi_c}(P)\sqrt{|P|} \rangle,
\end{equation*}
where $\omega(\chi_c)$ is the root number of $\chi_c$.
The large conductor limit gives the following non-constant function of $P$.

\begin{thm}\label{thm:cubic_Kummer}
Let $q\equiv 1\pmod{3}$ and $d = 2n$. For every $P \in \mathcal{P}_n$ and $\varepsilon > 0$, we have
\begin{equation*}
    \widetilde{M}_{q,d}(P) = \overline{\omega(\chi_P)} + O\left(q^{(\varepsilon-\frac16) d}\right)
\end{equation*}
as $d \to \infty$.
\end{thm}

\subsection{Discussion}\label{sec:discussion}
In this section, we outline the connection between the murmuration density and the average traces $\langle \operatorname{Tr} \Theta_D^n \rangle$. 
Taking the logarithmic derivative of the identity
\begin{equation*}
    \det(I-uq^{1/2} \Theta_D) = \prod_{P} (1-\chi_D(P)u^{\deg P})^{-1}
\end{equation*}
 and comparing coefficients of $u^n$, yields the explicit formula
\begin{equation*}
    \operatorname{Tr} \Theta_D^n = -\frac{1}{q^{n/2}} \sum_{f\in \mathcal{M}_n} \Lambda(f)\chi_D(f).
\end{equation*}
Separating out the higher prime powers and averaging over $\mathcal{H}_{2g+1}$ yields
\begin{equation*}
    \langle \operatorname{Tr} \Theta_D^n \rangle = -\frac{n}{q^{n}} \sum_{P\in \mathcal{P}_n} \langle \chi_D(P) \sqrt{|P|} \rangle - \frac{n}{q^{n/2}} \sum_{\substack{P\in \mathcal{P}\\\deg P^k = n \\ k \geq 2}} \langle \chi_D(P^k) \rangle,
\end{equation*}
and we recognize that when $n = 2g$ we have 
\begin{equation*}
    \frac{n}{q^{n}} \sum_{P\in \mathcal{P}_n} \langle \chi_D(P) \sqrt{|P|} \rangle = \mathbb{E}_{\mathcal{P}_{2g}}[M_{q,g}(P)] + O\left(\frac{1}{q^{n/2}}\right)
\end{equation*}
by the prime number theorem for $\mathbb{F}_q[T]$. 
The one-level density now follows easily from a Fourier expansion on $\mathrm{USp}(2g)$. 

Observe that the traces     $\langle \operatorname{Tr} \Theta_D^n \rangle$ and the one-level density only see the average value of the murmuration density at the transition range $n = 2g$. 
On the other hand, the murmuration density captures the contribution of individual primes to the zero-statistics of the family. 
In this sense, the murmuration density is a much finer statistical feature of the family than the one-level density. 
This connection to one-level density is more complicated in the case of non-self-dual families. One may be able to relate the murmuration density to a linear combination of the expectations $\langle \chi_c(P)\sqrt{|P|}\rangle$ taken over subfamilies along which the root number is fixed. One can of course consider the average $\langle \chi_c(P)\sqrt{|P|}\rangle$ over the whole family instead, but we will not pursue this here.

Next, we comment on the difference between our proof of the one-level density and the one given in \cite{wang2024noteszetaratiostabilization}. 
Wang's proof extends the idea in \cite{bergström2024hyperellipticcurvesscanningmap} to not only prove the moments for the family of quadratic characters, but to prove the ratios conjecture for this family, up to a restriction on the distance to the critical line. 
The ratios conjecture then captures the one-level density as explained by Andrade and Keating \cite{Andrade_2014}. 
Both proofs ultimately use the Grothendieck trace formula to convert the problem into cohomology and the stable range for $\mathcal{H}_g^{1,0}$ to compute the main term. 
As such, they are both limited by the error term obtained from bounding the unstable range, and this is the main obstruction to improvement. 
The difference is that our proof is more direct, expressing the trace averages $\langle \operatorname{Tr} \Theta_D^n \rangle$ in terms of cohomology and computing the contribution of the stable range. 
This simplifies the combinatorics that one must deal with, and it removes the need for any complex-analytic input. 
On the other hand, Wang establishes a more general result from which one can study $n$-level density.

Regarding cubic characters, to prove Theorem~\ref{thm:cubic_Kummer}, it suffices to obtain precise asymptotics for averages of cubic Gauss sums over function fields, which was done in \cite{DFL} with applications to non-vanishing. Number-field analogues of these results were established by G\"ulo\u{g}lu \cite{Guloglu2025NonVanishing}. Moreover, David and G\"ulo\u{g}lu \cite{DavidGuloglu2022OneLevelDensity} computed the one-level density past the support $(-1,1)$. We expect that similar techniques could be used to prove a number-field analogue of Theorem~\ref{thm:cubic_Kummer}, although we do not pursue this direction here. 

Finally, we provide a brief dictionary of how the scales involved translate to the number field setting where $\log p$ plays the role of $n$ and the logarithm of the conductor $\log N$ plays the role of $g$. 
The crudest statistic is the one-level density, which arises from the limit $n,g \to \infty$ (resp. $\log p, \log N \to \infty$) with $n/2g \to \tau$ (resp. $\log p/\log N \to \tau$). 
This agrees to first order with the random matrix theory prediction, although computing lower order terms reveals arithmetic corrections of order $g^{-1}$ (resp. $(\log N)^{-1}$) at the point $\tau = 1$. 
A finer statistic is captured by the traces $\langle \operatorname{Tr} \Theta_D^n \rangle$, which allows us to see the contribution of primes from a fixed degree $n$. 
In particular, this captures the arithmetic correction at $\tau = 1$ arising from primes $P \in \mathcal{P}_{2g}$ (resp. $p \asymp N$). 
Finer yet is the murmuration, where we have now zoomed into the point $\tau = 1$ and observe the contribution of individual (or small intervals of) primes $P \in \mathcal{P}_{2g}$ (resp. $p \sim yN$). 
Note that the result is strongest when one takes fewer primes, with a single prime as in \cite{HLOP} or \cite{Zubrilina} giving the ideal result.

\subsection*{Acknowledgements}
The authors gratefully acknowledge Peter Sarnak for suggesting that the correction term in Rudnick's theorem is closely related to murmurations, as well as many useful discussions throughout the course of this work. They also thank Dan Petersen and Victor Wang for clarifying discussions and useful feedback on an earlier draft, as well as Alexandra Florea for helpful discussions on cubic characters.

\section{Proof of Theorem~\ref{thm:quadratic_murmuration}}

The proof of Theorem~\ref{thm:quadratic_murmuration} mimics the analysis used by Rudnick in \cite[Section 4]{rudnick2008traces}. 
We take the following Lemmas from \cite{rudnick2008traces}:
\begin{lem}\label{lem:average_quad_characters} 
    Let $D\in\mathbb{F}_q[T]$ be monic of positive degree.
    For monic $P\in\mathbb{F}_q[T]$, we have:
    \begin{equation*}
        \langle \chi_D(P) \rangle = \frac{1}{(q-1)q^{2g}} \sum_{2\alpha + \beta =2g+1} \sigma(P;\alpha) \sum_{D \in \mathcal{M}_\beta} \left(\frac{D}{P}\right),
    \end{equation*}
    where
    \begin{equation*}
        \sigma(P;\alpha) = \sum_{\substack{A \in \mathcal{M}_\alpha\\ (A,P)=1}} \mu(A).
    \end{equation*}
\end{lem}
\begin{proof} See \cite[Section~3.2]{rudnick2008traces}.
\end{proof}
\begin{lem}\label{lem:sum_mobius}
    For $n \geq 2$ and a prime $P\in\mathbb{F}_q[T]$ of degree $n$, we have, for all $k \geq 1$:
    \begin{equation*}
        \sigma(P^k;\alpha) = \sigma_n(\alpha) = \begin{cases}
            1, & \text{ if } \alpha \equiv 0 \bmod n, \\
            -q, & \text{ if } \alpha \equiv 1 \bmod n, \\
            0, & \text{ else}.
        \end{cases}
    \end{equation*}
\end{lem}
\begin{proof} See \cite[Lemma 4]{rudnick2008traces}. \end{proof}
We also recall:
\begin{lem}[Quadratic reciprocity]\label{l.qr}
    If $D$ and $P$ are monic polynomials in $\F_q[T]$, then
    \begin{equation*}
        \left(\frac{D}{P}\right) = (-1)^{\frac{q-1}{2} \deg P \deg D} \left(\frac{P}{D}\right).
    \end{equation*}
\end{lem}
\begin{proof} See \cite[Theorem 3.3]{Rosen}.
\end{proof}

Let $P \in \mathbb{F}_q[T]$ be a monic irreducible polynomial of degree $n \in \mathbb{Z}_{\geq 0}$. 
We define the character sum:
\begin{equation*}
    S(\beta;P) := \sum_{
    D\in \mathcal{M}_\beta} \chi_D(P).
\end{equation*}
Applying Lemma~\ref{l.qr}, we can express $S(\beta;P)$ using coefficients $A_P(\beta) =\sum_{D\in \mathcal{M}_\beta} \chi_P(D)$ of the $L$-function $\mathcal{L}(u, \chi_P) = \sum_{\beta} A_P(\beta)u^\beta$:
\begin{equation*}
    S(\beta;P) = (-1)^{\frac{q-1}{2}\beta n}A_P(\beta).
\end{equation*}

\begin{lem}\label{lem:double_char_sum}
    Let $n = \deg P$. If $n \leq \beta$, then we have
    \begin{equation*}
        S(\beta;P) = 0.
    \end{equation*}
    If $n\in 2\mathbb{Z}_{\geq 1}$, with $\beta = n-1$, then we have
    \begin{equation*}
        S(\beta; P) = - q^{\frac{n-2}{2}}.
    \end{equation*}
\end{lem}
 \begin{proof} See the proof of \cite[Lemma 6]{rudnick2008traces}  and~\cite[Proposition 7]{rudnick2008traces}.
 \end{proof}
 We conclude the proof of Theorem~\ref{thm:quadratic_murmuration} by applying Lemma~\ref{lem:average_quad_characters}, Lemma~\ref{lem:sum_mobius}, and Lemma~\ref{lem:double_char_sum} for $n = 2g$ to equation~\eqref{eq.md} as follows:
\begin{align*}
    M_{q,g}(P) &= \frac{1}{(q-1)q^{2g}} \sum_{2\alpha + \beta =2g+1} \sigma(P;\alpha) \sum_{D\in \mathcal{M}_\beta} \left(\frac{D}{P}\right) \sqrt{|P|}  \\
    &= \frac{1}{(q-1)q^{g}} \sum_{2\alpha + \beta =2g+1} \sigma_n(\alpha)S(\beta;P) \\
    &=  \frac{1}{(q-1)q^{g}} \left(S(2g+1;P) - qS(2g-1;P)\right) \\
    &= \frac{1}{q-1}.
\end{align*}

\section{Proof of Theorem~\ref{thm:cohomological_formula}}
The proof proceeds by first expressing $\mathrm{Tr}(\Theta^n)$ as a sum of symplectic characters.
Using the Grothendieck trace formula, averages of these characters may be identified as traces of Frobenius acting on the \'{e}tale cohomology of a moduli space.
Subsequently, we estimate the sum via the contribution of the stable homology.
We will identify that the trivial character contributes the main random matrix theory term, as in \cite{katzsarnak1999}. 
Arithmetic corrections at small $n$ will be captured by sums over hooks of length $n$, which we bound for sufficiently large $q$ using the ideas in \cite{bergström2024hyperellipticcurvesscanningmap}.
Finally, the correction term will appear in a sum over hooks of the form $(a, 1^a)$ for $a \geq 1$, which we evaluate using \cite{bergström2024hyperellipticcurvesscanningmap}.

\subsection{Decomposition into irreducible symplectic characters}
Let $\mathbf{k}$ be a field of characteristic zero, let $V$ be a vector space over $\mathbf{k}$, and let $\Theta \in \mathrm{Sp}(V)$ be a symplectic linear map. 
Denote by $s_{\langle\lambda \rangle}$ the character of the irreducible representation $V_\lambda$ of $\mathrm{Sp}(V)$ of highest weight $\lambda$.  
\begin{lem}\label{lem:splitting_rule}
    Suppose that $n \in \mathbb{Z}_{\geq 1}$ and $\Theta \in \mathrm{Sp}_{2g}(\mathbf{k})$. We have
    \begin{equation*}
    \mathrm{Tr}(\Theta^n)
    =
    \begin{dcases}
    -\mathbbm{1}_{2\mid n} s_{\langle \emptyset \rangle}(\Theta)
    +
    \sum_{\ell=0}^{n-1}
    (-1)^\ell
    s_{\langle (n-\ell,1^\ell)\rangle}(\Theta),
    & \text{if } n \leq g, \\
    -\mathbbm{1}_{2\mid n} s_{\langle \emptyset \rangle}(\Theta)
    +
    \sum_{\ell=0}^{g-1}
    (-1)^\ell
    s_{\langle (n-\ell,1^\ell)\rangle}(\Theta)
    -
    \sum_{a=1}^{n-g-1}
    (-1)^{n+a}
    s_{\langle (a,1^{2g-n+a})\rangle}(\Theta),
    & \text{if } g < n \leq 2g, \\
    \sum_{\ell=0}^{g-1}
    (-1)^\ell
    s_{\langle (n-\ell,1^\ell)\rangle}(\Theta)
    -
    \sum_{a=0}^{g-1}
    (-1)^a
    s_{\langle (n-2g+a,1^a)\rangle}(\Theta),
    & \text{if } n > 2g.
    \end{dcases}
    \end{equation*}
\end{lem}

\begin{proof}
    It suffices to establish the equation after base change to $\overline{\mathbf{k}}$. 
    Furthermore, we may assume without loss of generality that
    \begin{equation}\label{eq.theta}
        \Theta = \mathrm{diag} (x_1,\ldots, x_g, x_1^{-1},\ldots,x_g^{-1}),
    \end{equation}
    with $x_i \neq x_j^{\pm 1}$ for all $i \neq j$ and $x_i \neq \pm1$ for all $i$. 
    Indeed, since we are establishing an identity of regular class functions on $\mathrm{Sp}_{2g}$, it is enough to check on the Zariski-dense set of regular semisimple elements in the split maximal torus. 
    With $\Theta$ as in equation~\eqref{eq.theta}, we have
    \begin{equation}\label{eq.Trp}
        \mathrm{Tr}(\Theta^n) = p_n(\Theta) = \sum_{i=1}^g x_i^n + x_i^{-n},
    \end{equation}
    where $p_n$ denotes the power sum polynomial of degree $n$ and we interpret $p_n(\Theta)$ as evaluating $p_n$ on the multiset of eigenvalues of $\Theta$. 
    As per \cite[2.6.8]{bergström2024hyperellipticcurvesscanningmap}, the Weyl character formula gives us an explicit bialternate-type formula
    \begin{equation}\label{eq.bialternate}
        s_{\langle \lambda \rangle}(\Theta) = \frac{\det\left(x_i^{\lambda_{j}+g-j+1}-x_i^{-(\lambda_j+g-j+1)}\right)_{1 \leq i,j \leq g}}{\det\left(x_i^{g-j+1}-x_i^{-(g-j+1)}\right)_{1 \leq i,j\leq g}}.
    \end{equation}
    Note that the denominator in equation~\eqref{eq.bialternate} is nonzero since we have passed to regular semisimple elements in the split maximal torus. 
    For $m \in \mathbb{Z}$ define the column vectors
    \begin{equation*}
        v_m = \begin{pmatrix}
            x_1^m - x_1^{-m} \\
            \vdots \\
            x_g^m - x_g^{-m}
        \end{pmatrix},
    \end{equation*}
    and let 
    \[\mathsf D = \det(v_g \mid v_{g-1} \mid \cdots \mid v_1).\] 
    Subsequently, equation~\eqref{eq.bialternate} gives
    \begin{equation}\label{eq.sthbia}
        s_{\langle \lambda \rangle}(\Theta) = \mathsf D^{-1}\det(v_{\lambda_1+g} \mid v_{\lambda_2 +g -1} \mid \cdots \mid v_{\lambda_g +1}).
    \end{equation}
    Applying equation~\eqref{eq.sthbia} in the case of a hook $\lambda = (a,1^b)$ with $a \geq 1$ and $0 \leq b \leq g-1$ gives
    \begin{align*}
        s_{\langle (a,1^b) \rangle}(\Theta) &= \mathsf D^{-1}\det\left(v_{a+g} \mid v_{g} \mid v_{g-1} \mid \cdots \mid v_{g-b+1} \mid v_{g-b-1} \mid \cdots \mid v_1\right) \\
        &= (-1)^b \mathsf D^{-1} \det\left(v_g  \mid \cdots \mid v_{g-b+1} \mid v_{g+a} \mid v_{g-b-1} \mid \cdots \mid v_1\right).
    \end{align*}
    Now we compute $p_n(\Theta)\mathsf D$ and compare. Let $H = \mathrm{diag}(x_i^n + x_i^{-n})_{1 \leq i \leq g}$ so $p_n(\Theta) =\mathrm{Tr}(H)$. Consider the equation
    \begin{equation}\label{eq.sumdet}
        \det((I+tH)v_g \mid \cdots \mid (I+tH)v_1) = \det(I+tH)\det(v_g \mid \cdots \mid v_1).
    \end{equation}
    Differentiating equation~\eqref{eq.sumdet} and evaluating at $t=0$, we deduce
    \begin{equation}\label{eq.diffsumdet0}
        \sum_{m=1}^g \det(v_g \mid \cdots \mid Hv_m \mid \cdots \mid v_1) = \mathrm{Tr}(H)\det(v_g \mid \cdots \mid v_1) = p_n(\Theta)\mathsf D.
    \end{equation}
    Noting that $Hv_m = v_{m+n} + v_{m-n}$, equation~\eqref{eq.diffsumdet0} implies that 
    \begin{equation}\label{eq.pnThe}
        p_n(\Theta) = \sum_{m=1}^g \frac{\mathsf D(m, m+n)}{\mathsf D} + \sum_{m=1}^g \frac{\mathsf D(m,m-n)}{\mathsf D}.
    \end{equation}
    where 
    \begin{equation*}
        \mathsf D(m,r) =  \det(v_g \mid \cdots \mid v_{r} \mid \cdots \mid v_1) 
    \end{equation*}
    and $v_{r}$ sits in place of $v_m$. 
    Notice that $\mathsf D(m,m+n)$ vanishes unless $m+n>g$. Setting $\ell = g-m$, the remaining terms are exactly
    \begin{equation*}
        \frac{\mathsf D(m,m+n)}{\mathsf D} = (-1)^\ell s_{\langle (n-\ell,1^\ell) \rangle}(\Theta), \ \ m+n\leq g.
    \end{equation*}
    Subsequently, we deduce that the first sum in equation~\eqref{eq.pnThe} contributes
    \begin{equation}\label{eq.L311}
        \sum_{\ell = 0}^{\min(n-1,g-1)} (-1)^\ell s_{\langle(n-\ell,1^\ell) \rangle }(\Theta).
    \end{equation}
    The second sum in equation~\eqref{eq.pnThe} is a bit more delicate. 
    If $m > n$, then $\mathsf D(m,m-n)$ vanishes by repeating a column.  
    If $m=n$ then $\mathsf D(m,m-n)$ vanishes since $v_0 = 0$. If $m < n$, then we have $v_{m-n} = - v_{n-m}$. If $n-m\leq g$, this gives a repeated column unless \(n-m=m\), 
that is, unless $n=2m$. 
Notice that if $n = 2m$ we get
    \begin{equation*}
        \frac{\mathsf D(m,m-n)}{\mathsf D} = \frac{\mathsf D(m,-m)}{\mathsf D} = -s_{\langle \emptyset \rangle}(\Theta) = -1, \ \ n=2m,
    \end{equation*}
    which yields a contribution of
    \begin{equation}\label{eq.L312}
        -\mathbbm{1}_{2 \mid n, n\leq 2g} s_{\langle\emptyset \rangle }(\Theta), 
    \end{equation}
    corresponding to the trivial representation. 
    If $n-m > g$ then set $a = n-g-m$ and $b =g-m$. 
    Since $a \geq 1$ and $0 \leq b \leq g-1$, we find
    \begin{equation*}
        \frac{\mathsf D(m,m-n)}{\mathsf D} = (-1)^{b+1}s_{\langle (a,1^b) \rangle}(\Theta), \ \ n-m>g. 
    \end{equation*}
    Thus we get a contribution 
    \begin{equation}\label{eq.L313}
        \sum_{a=\max(1,n-2g)}^{n-g-1} (-1)^{a-n+1} s_{\langle(a, 1^{2g-n+a})\rangle}(\Theta).
    \end{equation}
The Lemma follows by combining equations~\eqref{eq.Trp},~\eqref{eq.pnThe},~\eqref{eq.L311},~\eqref{eq.L312}, and~\eqref{eq.L313}.
\end{proof}

\subsection{Homological stability}
We refer the reader to \cite{bergström2024hyperellipticcurvesscanningmap} for more details regarding the homological concepts applied in this section. 
Let $H_g^{n,m}$ be the moduli space of hyperelliptic surfaces of genus $g$ with $n$ Weierstrass boundary components and
$m$ conjugate boundary pair components, defined in \cite[4.1.5]{bergström2024hyperellipticcurvesscanningmap}. 
$H_g^{n,m}$ is a topological space with algebro-geometric counterpart $\mathcal{H}_{g}^{n,m}$, that is, the moduli stack of smooth hyperelliptic curves of genus $g$, equipped with $n$ distinct ordered marked Weierstrass points, as well as $m$ distinct ordered non-Weierstrass points, none of which are conjugate under the hyperelliptic involution, and also equipped with a nonzero tangent vector at each marking. 
In particular, the analytification $\mathcal{H}_g^{n,m}(\mathbb{C})$ has the homotopy type of $H_g^{n,m}$.

We fix a choice of half Tate twist.
Let $\mathbb{V}$ be the local system on $\mathcal{H}_{g}^{n,m}$ given by $R^1p_! \overline{\mathbb{Q}}_\ell(1/2)$  where $p : \mathcal{C} \to \mathcal{H}_{g}^{n,m}$ is the universal $(n+2m)$-punctured curve. 
We will exploit a sequence of general comparison theorems for $\mathbb{V}$. 
From the work of Abramovich--Corti--Vistoli \cite{AbramovichCortiVistoli2003}, we have a smooth modular compactification of $\mathcal{H}_g^{n,m}$ over $\operatorname{Spec} \mathbb{Z}\left[\frac12\right]$ such that the complement is a normal crossing divisor. 
It follows that, for $q=p^r$ an odd prime power and a prime $\ell \neq p$, we have an isomorphism between étale cohomology in characteristic zero and positive characteristic: 
\begin{equation}\label{eq.isom1}
    H_k^{\mathrm{\acute{e}t}}(\mathcal{H}_g^{n,m} \otimes \overline{\mathbb{Q}}, \mathbb{Q}_\ell) \cong H_k^{\mathrm{\acute{e}t}}(\mathcal{H}_g^{n,m} \otimes \overline{\mathbb{F}}_p, \mathbb{Q}_\ell), 
\end{equation}
where we define $H_k^{\mathrm{et}}(X,\mathcal{F}) = H_{c}^{2 \dim X-k}(X,\mathcal{F}(\dim X))$ as per \cite{wang2024noteszetaratiostabilization}. 

Using the Artin comparison theorem and Poincaré duality, we pass to the singular homology of the topological space $H_g^{n,m}$:
\begin{equation}\label{eq.isom2}
    H^{\mathrm{sing}}_k(H_g^{n,m}, \mathbb{Q}_\ell) \cong H_k^{\mathrm{\acute{e}t}}(\mathcal{H}_g^{n,m} \otimes \overline{\mathbb{Q}}, \mathbb{Q}_\ell).
\end{equation}
Similarly, we get comparisons for coefficients in the local system $S^\lambda(\mathbb{V})$ where $S^\lambda$ is the Schur functor. 
Inside of $S^\lambda(\mathbb{V})$ we will consider the symplectic irreducible $\mathbb{V}_\lambda  \subseteq S^\lambda(\mathbb{V})$, which we assume is twisted to be pure of weight zero.
Specialising to the case $(n,m) = (1,0)$, a stalk of $\mathbb{V}$ at $D \in \mathcal{H}_{g}^{1,0}$ is given by 
\begin{equation*}
    \mathbb{V}_D = H_{\mathrm{\acute{e}t}}^1(C_D, \overline{\mathbb{Q}}_\ell)(1/2).
\end{equation*}
Since $\mathbb{V}_\lambda$ is constructed functorially inside of $S^\lambda(\mathbb{V})$, the comparison goes through for coefficients in $\mathbb{V}_\lambda$. 
Furthermore, $H_{g}^{1,0} \simeq \mathrm{Conf}_{2g+1}(D)$ is an Eilenberg–MacLane space $K(\beta_{2g+1}, 1)$ where $\mathrm{Conf}_{2g+1}(D)$ is the unordered configuration space and $\beta_{2g+1}$ is the Artin braid group on $2g+1$ strands. 
This gives us a further comparison to group homology:
\begin{equation}\label{eq.isom3}
    H_k(\beta_{2g+1}, \mathbb{Q}_\ell) \cong H_k^{\mathrm{sing}}(H_g^{1,0}, \mathbb{Q}_\ell). 
\end{equation}
Likewise, we can take coefficients in $\mathbb{V}_\lambda$, which corresponds to the integral reduced Burau representation $V_\lambda$ of $\beta_{2g+1}$. 
Combining equations~\eqref{eq.isom1},~\eqref{eq.isom2} and~\eqref{eq.isom3}, we extract the following sequence of comparisons.
\begin{thm}\label{thm:comparisons}
    For $q = p^r$ an odd prime power and a prime $\ell \neq p$ (used to define $\mathbb{V}_\lambda$), 
    \begin{equation*}
        H_k(\beta_{2g+1}, V_\lambda) \cong H_k^{\mathrm{sing}}(H_g^{1,0}, \mathbb{V}_\lambda) \cong H_k^{\mathrm{\acute{e}t}}(\mathcal{H}_{g}^{1,0} \otimes \overline{\mathbb{Q}}, \mathbb{V}_\lambda) \cong H_k^{\mathrm{\acute{e}t}}(\mathcal{H}_{g}^{1,0} \otimes \overline{\mathbb{F}}_q, \mathbb{V}_\lambda).
    \end{equation*}
\end{thm}

The main topological input will be the uniform homological stability for $H_g^{1,0}$.

\begin{thm}\label{thm:stable_range}
    We have a uniform stable range
    \begin{equation}\label{eq.uniformstable}
        H_k(\beta_n, V_\lambda) \cong H_{k}(\beta_{n+1},V_\lambda), \ \ k\leq An+B,
    \end{equation}
    for some $A,B\in\mathbb{R}$.
    Note that $V_\lambda$ denotes a different representation on each side of equation~\eqref{eq.uniformstable}.
\end{thm}
\begin{proof}
This is    \cite[Proposition 1.5]{miller2025uniformtwistedhomologicalstability}. 
\end{proof}
Concretely, it is known that one can take $A=\frac{1}{34}$ and $B=-\frac{35}{34}$ \cite[Proposition 1.5]{miller2025uniformtwistedhomologicalstability}.
Taking $n=2g+1$, we write 
$\theta(g)=\left\lfloor 2Ag+A+B\right\rfloor$ for the bound where the uniform stable range holds. 
We have to transfer results between $H_{g}^{1,0}$ and $H_g^{0,1}$, which is justified on account of the following Lemma.

\begin{lem}\label{lem:stable_agree}
    The stable homology of $H_g^{1,0}$ with coefficients in $S^{\lambda}(V)$ agrees with that of $H_g^{0,1}$.
\end{lem}

\begin{proof}
    See \cite[4.4.8]{bergström2024hyperellipticcurvesscanningmap}.
\end{proof}

Now we collect various vanishing results and bounds on the stable homology $H_k(H_{\infty}^{1,0}, \mathbb{V}_\lambda)$.

\begin{prop}\label{prop:vanishing_odd}
    If $n$ is odd and $|\lambda|$ is odd, then $H_{k}(\beta_{n}, V_\lambda)$ vanishes rationally for all $k$.
\end{prop}

\begin{proof}
Since $V_\lambda$ is a direct summand of $S^\lambda(V)$ and homology is additive under finite direct sums, the result follows from \cite[Proposition 4.3.6]{bergström2024hyperellipticcurvesscanningmap}.
\end{proof}

\begin{thm}\label{thm:vanishing_lambda1}
    If $\lambda_1 > |\lambda|/2$ then $H_k(H_\infty^{1,0}, \mathbb{V}_\lambda) = 0$ for all $k$.
\end{thm}
\begin{proof}
 The result follows from combining Lemma~\ref{lem:stable_agree} and \cite[Theorem 7.0.6]{bergström2024hyperellipticcurvesscanningmap}.
\end{proof}

\begin{thm}\label{thm:vanishing_ell_lambda}
    If $k < \ell(\lambda)/2$ then $H_k(H_\infty^{1,0}, \mathbb{V}_\lambda) = 0$.
\end{thm}
\begin{proof}
 The result follows from combining Lemma~\ref{lem:stable_agree} and \cite[Theorem 7.0.14]{bergström2024hyperellipticcurvesscanningmap}.
\end{proof}

\begin{thm}\label{thm:purity}
The stable homology group $H_k^{\mathrm{\acute{e}t}}(\mathcal{H}_\infty^{1,0} \otimes \overline{\mathbb{F}}_q, \mathbb{V}_\lambda)$ is pure Tate of weight $-2k$.
\end{thm}

\begin{proof}
By~\cite[Theorem 9.2.2]{bergström2024hyperellipticcurvesscanningmap}, we know that 
$H_k^{\mathrm{\acute{e}t}}(\mathcal{H}_{\infty}^{0,1} \otimes \overline{\mathbb{Q}}, S^\lambda(\mathbb{V}))$ is pure Tate of weight $-2k$ with our normalization. 
The same follows for  $H_k^{\mathrm{\acute{e}t}}(\mathcal{H}_{\infty}^{0,1} \otimes \overline{\mathbb{Q}}, \mathbb{V}_\lambda)$ since $\mathbb{V}_\lambda$ is a Galois-stable direct summand of $S^\lambda(\mathbb{V})$. 
The comparison theorem is Galois-compatible, so we may pass to $H_k^{\mathrm{\acute{e}t}}(\mathcal{H}_{\infty}^{0,1} \otimes \overline{\mathbb{F}}_q, \mathbb{V}_\lambda)$ by Theorem~\ref{thm:comparisons}. 
Finally, we apply Lemma~\ref{lem:stable_agree} to complete the proof.
\end{proof}

\subsection{Unstable range}

For the unstable homology, we use the following bounds.

\begin{lem}\label{lem:dim_bound_Braid}
If $V$ is a representation of the Braid group $\beta_n$, then 
    \begin{equation*}
        \dim H_k(\beta_n, V) \leq {n-1 \choose k} \dim V.
    \end{equation*}
\end{lem}

\begin{proof}
    See \cite[Lemma 11.3.13]{bergström2024hyperellipticcurvesscanningmap}.
\end{proof}

\begin{lem}\label{lem:dim_bounds_hook}
    Let $V_\lambda = S^\lambda_{\mathrm{Sp}}(V)$ be the symplectic irreducible representation with highest weight $\lambda$, where $\lambda = (a, 1^b)$ is a hook with $a + b = n$ and $\dim V = 2g$. Then
    \begin{equation}\label{eq.dimV}
        \dim V_\lambda \leq (2g) 2^{4g+a-1}.
    \end{equation}
\end{lem}

\begin{proof}
    We have $\dim V_\lambda \leq \dim S^\lambda(V)$ and, by the hook-length formula,  
    \begin{equation}\label{eq.hlf}
        \dim S^\lambda(V) = \frac{(2g+a-1)!}{(2g-b-1)!n(a-1)!b!} = \frac{2g-b}{n} {2g+a-1 \choose a-1} {2g \choose b}.
    \end{equation}
Equation~\eqref{eq.dimV} follows from combining equation~\eqref{eq.hlf} with the following bounds:
    \begin{equation*}
        {2g+a-1 \choose a-1} \leq 2^{2g+a-1}, 
        \quad {2g \choose b} \leq 2^{2g}.
    \end{equation*}
\end{proof}

\begin{lem}\label{lem:mixed_weight}
    The \'etale homology
    \begin{equation*}
        H_k^{\mathrm{\acute{e}t}}(\mathcal{H}_g^{1,0} \otimes \overline{\mathbb{F}}_p, \mathbb{V}_\lambda),
    \end{equation*}
    is mixed of weight $\leq -k$.
\end{lem}

\begin{proof}
 Recalling that $\mathbb{V}_\lambda$ is pure of weight zero, the twist $\mathbb{V}_\lambda(d)$ is pure of weight $-2d$. Apply Deligne's purity theorem to conclude that 
    \begin{equation*}
        H_k^{\mathrm{\acute{e}t}}(\mathcal{H}_g^{1,0} \otimes \overline{\mathbb{F}}_p, \mathbb{V}_\lambda)=H_c^{2 d - k}(\mathcal{H}_g^{1,0} \otimes \overline{\mathbb{F}}_p, \mathbb{V}_\lambda(d)) 
    \end{equation*}
    is mixed of weight $\leq -k$.
\end{proof}

\subsection{Stable range}

Let $\Lambda := \varprojlim \mathbb{Z}[x_1,\ldots, x_n]^{S_n}$ be the ring of symmetric functions, and consider the ring of power series in $z$ with coefficients in $\Lambda$, i.e. the completion of $\Lambda[z]$. 
We denote by $\mathbf{Exp}$ and $\mathbf{Log}$ the plethystic exponential and plethystic logarithm respectively, as defined in \cite[2.3.9]{bergström2024hyperellipticcurvesscanningmap}. 
To evaluate the contribution of the stable homology, we will use the following central theorem:

\begin{thm}\label{thm:poincare_series}
We have
    \begin{equation}\label{eq:poincare_series}
        \sum_k \sum_\lambda \dim H_k(\mathcal{H}^{1,0}_\infty, \mathbb{V}_\lambda)(-z)^k s_{\lambda'} = \mathbf{Exp}\left(z^{-1}\mathbf{Log}\left(z + \sum_{r \geq 0} z^{r} h_{2r}\right)-1-h_2\right),
    \end{equation}
    where $\lambda'$ is the conjugate partition of $\lambda$.
\end{thm}

\begin{proof}
    See~\cite[Theorem 4.4.12]{bergström2024hyperellipticcurvesscanningmap}.
\end{proof}

Below we identify the stable homology of the trivial representation.

\begin{lem}\label{lem:trivial_rep_contribution}
    We have 
    \begin{equation*}
        \sum_{k \geq 0} (-1)^k \mathrm{Tr}\left(\mathrm{Frob}_q \mid H_k^{\mathrm{\acute{e}t}}(\mathcal{H}_{\infty}^{1,0} \otimes \overline{\mathbb{F}}_q, \mathbb{V}_\emptyset)\right) = 1-q^{-1}.  
    \end{equation*}
\end{lem}

\begin{proof}
    By 1.3.7 in \cite{bergström2024hyperellipticcurvesscanningmap}, we have $H^*(\beta_n,\Q)\cong H^*(S^1,\Q)$, which implies
    \begin{equation*}
        \dim H_k(\beta_\infty, V_\emptyset) = \begin{cases}
            1 & \text{ if } k \leq 1, \\
            0 & \text{ if } k > 1,
        \end{cases}
    \end{equation*}
    and we have that $H_k(\beta_\infty, V_\emptyset)$ are pure Tate twists by Theorem~\ref{thm:purity}.
\end{proof}

Next, we will extract from Theorem~\ref{thm:poincare_series} the contribution of stable homology coming from hook representations. 
Given an element $f \in \Lambda$ we define the plethystic evaluation on a virtual alphabet $x-y$, denoted $f[x-y]$, as the ring homomorphism $\Lambda \to \mathbb{Z}[x,y]$ defined by $p_r[x-y] \mapsto x^r-y^r$. 
It is easy to check that plethystic evaluation commutes with the Adams operation, defined in \cite[2.2.13]{bergström2024hyperellipticcurvesscanningmap}, and thus commutes with $\mathbf{Exp}$ and $\mathbf{Log}$.

\begin{lem}\label{lem:non_hook_vanish}
    Let $\lambda$ be a non-empty partition. We have
    \begin{equation*}
        s_\lambda[x-y] = s_{\lambda'}[x-y] = 0,
    \end{equation*}
    if $\lambda$ is not a hook, and 
    \begin{equation*}
        s_{\lambda}[x-y] = (x-y)x^{a-1}(-y)^b,
    \end{equation*}
    if $\lambda = (a,1^b)$ is a hook with $a \geq 1$.

\end{lem}

\begin{proof}
    We use the idea in \cite[Ch.I, \S5, Example 10]{macdonald1995symmetric}. By Jacobi--Trudi formula we have
    \begin{equation*}
        s_\lambda = \det(h_{\lambda_i-i+j})_{1 \leq i,j \leq n}.
    \end{equation*}
    By \cite[Ch.~I, \S3, Example~23]{macdonald1995symmetric} and \cite[Ch.~I, \S2, Equation~2.6]{macdonald1995symmetric} we find that
    \begin{equation*}
        \sum_{r \geq 0} h_r[x-y]t^r = \frac{1-yt}{1-xt},
    \end{equation*}
   which implies that $h_0[x-y] = 1$ and 
    \begin{equation}\label{eq:h_r_expression}
        h_r[x-y] = x^{r-1}(x-y),  \quad (r \geq 1).
    \end{equation}
    Equation~\eqref{eq:h_r_expression} implies $h_r[1-t] = 1-t$ for $r \geq 1$. 
    Moreover, by homogeneity of $s_\lambda$,  we have
    \begin{equation*}
        s_\lambda[x-y] = x^{|\lambda|} s_\lambda[1-y/x],
    \end{equation*}
    and so it suffices to consider the specialization $h_r \mapsto X$ for $r \geq 1$ where $X = 1-y/x$. 
        If $\lambda$ is not a hook then $\lambda_1 \geq \lambda_2 \geq 2$, so
    \begin{equation*}
        \lambda_1-1+j \geq 1, \quad \lambda_2-2+j \geq 1.
    \end{equation*}
    In particular, the first two rows of the Jacobi--Trudi matrix are $(X,\ldots,X)$, and we conclude that $s_\lambda[x-y] = 0$. Otherwise, let $\mathsf D_b$ be the determinant of the $(b+1)\times(b+1)$ matrix
    \begin{equation*}
        \mathsf D_b = \det \begin{pmatrix}
            X & X & X &\cdots & X \\
            1 & X & X &\cdots & X \\
            0 & 1 & X &\cdots& X \\
            \vdots & \vdots & \vdots & \ddots & \vdots \\
            0 & 0 & 0 & \cdots & X
        \end{pmatrix} .
    \end{equation*}
    Expanding $\mathsf D_b$ along the first column gives
    \begin{equation*}
        \mathsf D_b = X\mathsf D_{b-1} - \mathsf D_{b-1} = (X-1)\mathsf D_{b-1},
    \end{equation*}
    and $\mathsf D_0 = X$. Hence $\mathsf D_b = X(X-1)^b$ and we conclude
    \begin{equation*}
        s_{(a,1^b)}[x-y] = x^{a+b}s_{(a,1^b)}\left[1-\frac{y}{x}\right] = x^{a+b}\left(1- \frac{y}{x}\right)\left(-\frac{y}{x}\right)^{b} = (x-y)x^{a-1}(-y)^{b}.
    \end{equation*} 
\end{proof}

\begin{lem}\label{lem:series_LHS}
    We have
    \begin{equation*}
        \mathbf{Exp}\left(z^{-1}\mathbf{Log}\left(z + \sum_{r \geq 0} z^{r} h_{2r}\right)-1-h_2\right)[t-t^{-1}] = 1-zt^2.
    \end{equation*}
\end{lem}

\begin{proof}
    Note that plethystic evaluation commutes with $\mathbf{Exp}$ and $\mathbf{Log}$. 
    We have $h_0[t-t^{-1}] = 1$ and, applying equation~\eqref{eq:h_r_expression}, we deduce
    \begin{equation*}
       h_{2r}[t-t^{-1}] = t^{2r} - t^{2r-2}, \quad (r \geq 1).
    \end{equation*}
    Thus 
    \begin{equation*}
        z + \sum_{r \geq 0} z^r h_{2r}[t-t^{-1}] =1 + z + \sum_{r \geq 1} z^r(t^{2r}-t^{2r-2}) = \frac{1-z^2t^2}{1-zt^2}.
    \end{equation*}
    For a monomial $x$, we compute
    \begin{equation*}
        \mathbf{Exp}(x) = \exp\left(\sum_{k \geq 1} \frac{x^k}{k}\right) = \exp(-\log(1-x)) = \frac{1}{1-x},
    \end{equation*}
    and so, since $\mathbf{Exp}$ is additive,
    \begin{equation*}
        \mathbf{Exp}(a-b) = \mathbf{Exp}(a)\mathbf{Exp}(-b) = \frac{1-b}{1-a}.
    \end{equation*}
    In particular, we have
    \begin{equation*}
        \mathbf{Log}\left(\frac{1-z^2t^2}{1-zt^2}\right) = zt^2-z^2t^2.
    \end{equation*}
    Since $h_2[t-t^{-1}] = t^2-1$, the left-hand side simplifies to
    \begin{equation*}
        \mathbf{Exp}(-zt^2) = 1-zt^2.
    \end{equation*}
\end{proof}

\begin{cor}\label{cor:series_evaluation}
    For $\Delta \in \mathbb{Z}_{\geq 0}$, we have
    \begin{equation*}
         \sum_{\ell \geq 1}  (-1)^\ell  \sum_{k\geq 0}(-z)^k \dim_k  H_k(\mathcal{H}^{1,0}_\infty, \mathbb{V}_{(\ell, 1^{\Delta+\ell})}) = \begin{cases} z, & \text{ if } \Delta = 0,\\
         0, & \text{ otherwise}.
         \end{cases}
    \end{equation*}
\end{cor}

\begin{proof}
    Evaluating equation~\ref{eq:poincare_series} at $t-t^{-1}$ and applying Lemmas~\ref{lem:trivial_rep_contribution},~\ref{lem:non_hook_vanish}, and ~\ref{lem:series_LHS} yields:
    \begin{equation*}
        (1-z)+(t^{-1}-t)\sum_{\Delta \geq 0}\sum_{\ell \geq 1} (-1)^{\ell} \sum_{k\geq 0} (-z)^k \dim_k H_k(\mathcal{H}_{\infty}^{1,0}, \mathbb{V}_{(\ell,1^{\Delta+\ell})})t^{\Delta+1}= 1-zt^2
    \end{equation*}
    since the conjugate of $(a,1^b)$ is $ (b+1,1^{a-1})$ and the trivial representation contributes a term $1-z$. Therefore
    \begin{equation*}
        \sum_{\Delta \geq 0}\sum_{\ell \geq 1} (-1)^{\ell} \sum_{k \geq 0} (-z)^k \dim_k H_k(\mathcal{H}_{\infty}^{1,0}, \mathbb{V}_{(\ell,1^{\Delta+\ell})})t^{\Delta+1} = zt.
    \end{equation*}
\end{proof}

\begin{lem}\label{lem:necklace}
    We have an identity
    \begin{multline*}
        \mathbf{Exp}\left(z^{-1}\mathbf{Log}\left(z + \sum_{r \geq 0} z^{r} h_{2r}\right)-1-h_2\right)[t-(-t)] \\
        = (1-t^2)^2(1-z)\prod_{\substack{n=1 \\ 2 \nmid n}}^\infty\left(1+ \frac{2z^nt^{2n}}{(1+z^n)(1-z^nt^{2n})}\right)^{i_n(z^{-1})}.
    \end{multline*}
    where $i_n(t) = \frac{1}{n}\sum_{d \mid n} \mu(n/d)t^d$ is the $n$-th necklace polynomial.
\end{lem}

\begin{proof}
    By \cite[Proposition 7.0.17]{bergström2024hyperellipticcurvesscanningmap} and additivity of $\mathbf{Exp}$ we have 
    \begin{multline*}
        \mathbf{Exp}\left(z^{-1}\mathbf{Log}\left(z + \sum_{r \geq 0} z^{r} h_{2r}\right)-1-h_2\right) \\
        = \mathbf{Exp}(-h_2)(1-z)\prod_{n=1}^\infty \left(1 + \frac{1}{1+z^n}\sum_{k > 0} z^{nk} \psi_n(h_{2k})\right)^{i_n(z^{-1})},
    \end{multline*}
    where $\psi_n$ is the Adams operation. Evaluating at $t-(-t)$ yields
    \begin{equation*}
        \mathbf{Exp}(-h_2)[t-(-t)] = \mathbf{Exp}(-2t^2) = (1-t^2)^2,
    \end{equation*}
    and
    \begin{equation*}
        \psi_n(h_{2k})[t-(-t)] = h_{2k}[t^n-(-t)^n] = \begin{cases}
            2t^{2nk} & \text{ if } 2 \nmid n, \\
            0 & \text{ if }2 \mid n.
        \end{cases}
    \end{equation*}
    The result follows from evaluating the geometric series.
\end{proof}

\begin{cor}\label{cor:series_bound}
    For $q > 16$ we have 
   \begin{equation*}
       \sum_{\ell\geq 0} \sum_{k \geq 0}   \dim H_k(\mathcal{H}_\infty^{1,0} \otimes \overline{\mathbb{F}}_q, \mathbb{V}_{(n-\ell,1^\ell)}) q^{-k/2} \ll 2^{n/2},
   \end{equation*}
   where the implied constant is absolute. 
\end{cor}

\begin{proof}
    Evaluating equation \eqref{eq:poincare_series} at $t - (-t)$ and applying  Lemmas~\ref{lem:trivial_rep_contribution},~\ref{lem:non_hook_vanish}, and \ref{lem:necklace} yields:
    \begin{multline*}
        (1-z) + 2\sum_{n \geq 1} \sum_{\ell =0}^{n-1} \sum_{k\geq0} \dim H_k(\mathcal{H}_\infty^{1,0}, \mathbb{V}_{(n-\ell,1^\ell)})(-z)^k t^{n} \\
        = (1-t^2)^2(1-z)\prod_{\substack{n=1 \\ 2 \nmid n}}^\infty\left(1+ \frac{2z^nt^{2n}}{(1+z^n)(1-z^nt^{2n})}\right)^{i_n(z^{-1})},
    \end{multline*}
    where we have used that the conjugate of $(n-\ell,1^\ell)$ is $(\ell +1, 1^{n-\ell-1})$. 
    We seek the coefficient of $t^n$ evaluated at $z = -q^{-1/2}$. 
    If $q > 16$ and $|t| = 2^{-1/2}$, for $n$ odd we have
    \begin{equation*}
        \frac{2z^nt^{2n}}{(1+z^n)(1-z^nt^{2n})} \leq \frac{2q^{-n/2}2^{-n}}{(1-q^{-n/2})(1-q^{-n/2}|t|^{2n})} \leq 2\left(1-\frac{1}{4}\right)^{-1}\left(1-\frac{1}{8}\right)^{-1} q^{-n/2}2^{-n}.
    \end{equation*}
    Now we bound for $q > 16$
    \begin{equation*}
        i_n(z^{-1}) = \frac{1}{n}\sum_{d \mid n}\mu(n/d)q^{d/2} \leq \frac{4}{3} \frac{q^{n/2}}{n}.
    \end{equation*}
   Therefore
    \begin{align*}
        \sum_{\substack{n \geq 1 \\ 2 \nmid n}} \left|i_n(z^{-1}) \log\left(1+ \frac{2z^nt^{2n}}{(1+z^n)(1-z^nt^{2n})}\right) \right| \ll \sum_{\substack{n \geq 1 \\ 2 \nmid n}} \frac{2^{-n}}{n} \ll 1,
    \end{align*}
    and exponentiation yields a uniform bound
    \begin{equation*}
         \prod_{\substack{n=1 \\ 2 \nmid n}}^\infty\left(1+\frac{2z^nt^{2n}}{(1+z^n)(1-z^nt^{2n})} \right)^{i_n(z^{-1})} \ll 1.
    \end{equation*}  
    By Cauchy's estimate, we conclude that the $t^n$ coefficient of 
    \begin{equation*}
        (1-z)\prod_{\substack{n=1 \\ 2 \nmid n}}^\infty\left(1+ \frac{2z^nt^{2n}}{(1+z^n)(1-z^nt^{2n})}\right)^{i_n(z^{-1})}
    \end{equation*}
    at $z = -q^{-1/2}$ is $\ll 2^{n/2}$. 
    The $t^n$ coefficient of the expression we are interested in will be a fixed linear combination of $t^{n}, t^{n-2}$ and $t^{n-4}$ coefficients of this product. 
\end{proof}

\subsection{Traces of Frobenius class}

For a genus $g$ and a weight $\lambda$, we introduce the notation
\begin{equation*}
    \mathrm{Tr}_\lambda(g) := \frac{1}{q^{2g+1}} \sum_{d \in \mathcal{H}_{2g+1}(\mathbb{F}_q)} s_{\langle \lambda \rangle}(\Theta_d).
\end{equation*}
Using Lemma~\ref{lem:splitting_rule} and the fact that $\#\mathcal{H}_{2g+1}=(q-1)q^{2g}$, we can express
\begin{equation}\label{eq.TrtDn}
    \langle \operatorname{Tr}(\Theta^n_D) \rangle = \frac{q}{q-1}\sum_{\lambda \in P(n)} c_{\lambda,n,g} \mathrm{Tr}_\lambda(g),
\end{equation}
for some $ c_{\lambda,n,g}  \in \mathbb{Z}$ and a set $P(n)$ of partitions, potentially including the empty partition, depending only on $n$. 
By the Grothendieck trace formula, equation~\eqref{eq.TrtDn} becomes
\begin{equation}\label{eq:trace_formula}
     \langle \operatorname{Tr}(\Theta^n_D) \rangle = \frac{q}{q-1} \sum_{\lambda \in P(n)} c_{\lambda,n,g}  \sum_{k \geq 0}  (-1)^k \mathrm{Tr}(\mathrm{Frob}_q \mid H_k^{\mathrm{\acute{e}t}}(\mathcal{H}_g^{1,0} \otimes \overline{\mathbb{F}}_q, \mathbb{V}_\lambda)).
\end{equation}

Now we bound the contribution of the hooks of length $n$.
\begin{prop}\label{prop:small_n_contribution}
    If $2 \nmid q$ and $q > 16$, then
    \begin{multline*}
        \sum_{\ell = 0}^{\min(n-1,g-1)} (-1)^\ell \sum_{k \geq 0}(-1)^k  \mathrm{Tr}\left(\mathrm{Frob}_q \mid H_k^{\mathrm{\acute{e}t}}(\mathcal{H}_{g}^{1,0} \otimes \overline{\mathbb{F}}_q, \mathbb{V}_{(n-\ell,1^\ell)})\right) \\
        \ll \mathbbm{1}_{n \leq g}  2^{n/2} q^{-n/8} + g2^{6g+n} q^{-\theta(g)/2},
    \end{multline*}
    where the implied constant is absolute and 
    $\theta(g)=\lfloor 2Ag+A+B\rfloor$ is the stable range of Theorem \ref{thm:stable_range}.

\end{prop}

\begin{proof}
    Let $L = \min(n-1,g-1)$. Using Lemmas \ref{lem:dim_bound_Braid},~\ref{lem:dim_bounds_hook}, and~\ref{lem:mixed_weight}, we get
    \begin{align*}
        &\sum_{\ell = 0}^{L} \sum_{k > \theta(g)}(-1)^{k+\ell} \mathrm{Tr}(\mathrm{Frob}_q \mid H_k^{\mathrm{\acute{e}t}}(\mathcal{H}_g^{1,0} \otimes \overline{\mathbb{F}}_q, \mathbb{V}_{(n-\ell,1^\ell)}))  \\
        & 
        \leq     \sum_{\ell = 0}^{L}  \sum_{k > \theta(g)} \dim H_k^{\mathrm{\acute{e}t}}(\mathcal{H}_g^{1,0} \otimes \overline{\mathbb{F}}_q, \mathbb{V}_{(n-\ell,1^\ell)})) q^{-k/2} \\
        &\leq \left(\sum_{\ell = 0}^{L} \dim  \mathbb{V}_{(n-\ell,1^\ell)}) \right)\left(  \sum_{k > \theta(g)} {2g \choose k}  q^{-k/2}\right) \\
        &\leq (2g)2^{4g+n}\left(\sum_{\ell = 0}^{L} 2^{-\ell-1}  \right)\left(  \sum_{k > \theta(g)} {2g \choose k}  q^{-k/2}\right) \\
        &\leq (2g)2^{6g+n} q^{-\theta(g)/2}.
    \end{align*}
    The stable range in this case gives a noticeable contribution when $n$ is small, but this decays rapidly with $n$. Note that this corresponds to the arithmetic contribution at zero in the one-level density. By Theorems~\ref{thm:stable_range}, \ref{thm:vanishing_lambda1}, and \ref{thm:vanishing_ell_lambda}, since the finite-genus homology agrees with the stable homology for $k\leq \theta(g)$,
    we have
    \begin{align*}
        &\sum_{\ell = 0}^{L} (-1)^\ell \sum_{k = 0}^{\theta(g)} (-1)^k  \mathrm{Tr}(\mathrm{Frob}_q \mid H_k^{\mathrm{\acute{e}t}}(\mathcal{H}_g^{1,0} \otimes \overline{\mathbb{F}}_q, \mathbb{V}_{(n-\ell,1^\ell)}))\\
        &=\sum_{\ell = 0}^{L} (-1)^\ell \sum_{k = 0}^{\theta(g)} (-1)^k  \mathrm{Tr}(\mathrm{Frob}_q \mid H_k^{\mathrm{\acute{e}t}}(\mathcal{H}_\infty^{1,0} \otimes \overline{\mathbb{F}}_q, \mathbb{V}_{(n-\ell,1^\ell)}))\\
        &=\sum_{\ell = \lceil n/2\rceil}^{L} (-1)^\ell \sum_{k = \lceil (\ell+1)/2\rceil}^{\theta(g)} (-1)^k  \mathrm{Tr}(\mathrm{Frob}_q \mid H_k^{\mathrm{\acute{e}t}}(\mathcal{H}_\infty^{1,0} \otimes \overline{\mathbb{F}}_q, \mathbb{V}_{(n-\ell,1^\ell)})).
    \end{align*}
    If $n > g$ then, assuming we consider a stable range with $A < 1/8$\footnote{In \cite[1.1.15]{bergström2024hyperellipticcurvesscanningmap} it is conjectured that one can take $A = 1/4$, but one can always use smaller $A$.}, we have 
    \begin{equation*}
        k >\frac{\ell}{2} \geq \frac{n}{4} > \frac{g}{4} > \theta(g),
    \end{equation*}
    so we conclude that the stable range vanishes\footnote{This argument works for any 
    $n>8Ag+4A+4B$.}.
    If $n \leq g$ we use Corollary~\ref{cor:series_bound} to conclude
    \begin{align*}
         &\sum_{\ell = \lceil n/2\rceil}^{L} (-1)^\ell \sum_{k = \lceil (\ell+1)/2\rceil}^{\theta(g)} (-1)^k  \mathrm{Tr}(\mathrm{Frob}_q \mid H_k^{\mathrm{\acute{e}t}}(\mathcal{H}_\infty^{1,0} \otimes \overline{\mathbb{F}}_q, \mathbb{V}_{(n-\ell,1^\ell)}))\\
        & \leq  \sum_{\ell = \lceil n/2\rceil}^{L} \sum_{k = \lceil (\ell+1)/2\rceil}^{\theta(g)}  \dim H_k^{\mathrm{\acute{e}t}}(\mathcal{H}_{\infty}^{1,0} \otimes \overline{\mathbb{F}}_q, \mathbb{V}_{(n-\ell,1^\ell)})q^{-k} \\
        &\leq q^{-n/8} \sum_{\ell\geq 0} \sum_{k \geq 0}   \dim H_k^{\mathrm{\acute{e}t}}(\mathcal{H}_\infty^{1,0} \otimes \overline{\mathbb{F}}_q, \mathbb{V}_{(n-\ell,1^\ell)}) q^{-k/2} \\
        & \ll 2^{n/2} q^{-n/8},
    \end{align*}
    for $q > 16$. 
\end{proof}

\begin{prop}\label{prop:correction_term}
    If $g<n \leq 2g$, 
    $q>2^{3/A+2}$, then we have
    \begin{equation*}
        \sum_{a=1}^{n-g-1} (-1)^{n+a} \sum_{k \geq 0} (-1)^k \mathrm{Tr}(\mathrm{Frob}_q \mid H_k^{\mathrm{\acute{e}t}}(\mathcal{H}_{g}^{1,0} \otimes \overline{\mathbb{F}}_q ,\mathbb{V}_{ (a,1^{2g-n+a}) })) = \frac{\mathbbm{1}_{n=2g}}{q} + O\left(g2^{5g+n}q^{-\theta(g)/2}\right).
    \end{equation*}
\end{prop}

\begin{proof}
    We bound the unstable range as in the proof of Proposition~\ref{prop:small_n_contribution}, where we find an error of size
    \begin{align*}
        (2g)2^{4g}\left(\sum_{a=1}^{n-g-1} 2^{a-1}\right)\left(\sum_{k > \theta(g)} {2g \choose k} q^{-k/2}\right) \leq (2g)2^{5g+n} q^{-\theta(g)/2}.
    \end{align*}
    The stable range contributes $q^{-1}\mathbbm{1}_{n=2g}$ by Corollary~\ref{cor:series_evaluation}, and by Theorem~\ref{thm:vanishing_ell_lambda} we get an error of size
    \begin{align*}
        &\sum_{a \geq n-g} (-1)^{n+a} \sum_{k \geq 0} (-1)^k \mathrm{Tr}(\mathrm{Frob}_q \mid H_k^{\mathrm{\acute{e}t}}(\mathcal{H}_{\infty}^{1,0} \otimes \overline{\mathbb{F}}_q ,\mathbb{V}_{ (a,1^{2g-n+a})}))  \\
        \leq & \sum_{a \geq n-g} \sum_{k \geq \lceil (2g-n+a+1)/2\rceil} |\mathrm{Tr}(\mathrm{Frob}_q \mid H_k^{\mathrm{\acute{e}t}}(\mathcal{H}_{\infty}^{1,0} \otimes \overline{\mathbb{F}}_q ,\mathbb{V}_{ (a,1^{2g-n+a})}))| \\
        \leq & \sum_{a \geq n-g}  \sum_{k \geq \lceil (2g-n+a+1)/2\rceil}  |\mathrm{Tr}(\mathrm{Frob}_q \mid H_k^{\mathrm{\acute{e}t}}(\mathcal{H}_{g(k)}^{1,0} \otimes \overline{\mathbb{F}}_q ,\mathbb{V}_{(a,1^{2g-n+a})}))|,
    \end{align*}
    where
    \begin{equation*}
        g(k) =\left\lceil\frac{k-A-B}{2A}\right\rceil
    \end{equation*}
    so that 
    \begin{equation*}
        H_k^{\mathrm{\acute{e}t}}(\mathcal{H}_{g(k)}^{1,0} \otimes \overline{\mathbb{F}}_q ,\mathbb{V}_{(a,1^{2g-n+a})}) = H_k^{\mathrm{\acute{e}t}}(\mathcal{H}_{\infty}^{1,0} \otimes \overline{\mathbb{F}}_q ,\mathbb{V}_{(a,1^{2g-n+a})}).
    \end{equation*}
    By Lemma~\ref{lem:dim_bound_Braid} and Lemma~\ref{lem:dim_bounds_hook} we have
    \begin{equation*}
        \dim H_k^{\mathrm{\acute{e}t}}(\mathcal{H}_{g(k)}^{1,0} \otimes \overline{\mathbb{F}}_q ,\mathbb{V}_{(a,1^{2g-n+a}) }) \leq  (2g(k))2^{6g(k)+a-1},
    \end{equation*}
    so by Theorem~\ref{thm:purity} we have an error of size
    \begin{align*}
        &\sum_{a \geq n-g} \sum_{k \geq \lceil (2g-n+a+1)/2\rceil} (2g(k))2^{6g(k)+a-1} q^{-k} \\
        \ll & \sum_{a \geq n-g} 2^{a-1} \sum_{k \geq K_a } (k+1) \rho^k \\
        \ll &  \sum_{a \geq n-g} 2^{a-1} (K_a+1) \rho^{K_a},
    \end{align*}
    where $K_a = \lceil (2g-n+a+1)/2\rceil$ and 
    $\rho=2^{3/A}/q$. If $\frac32 g < n \leq 2g$, then $a \geq n-g > \frac12 g$. Here, the error is of size
    \begin{align*}
         &\sum_{a \geq g/2} 2^{a-1} (K_a+1) \rho^{K_a} \ll \sum_{a \geq g/2} 2^{a-1} (g+a+1) \rho^{\lceil a/2\rceil} \ll \sum_{a \geq g/2} (g+a+1) (4\rho)^{a/2} \ll g(4\rho)^{g/4}.
    \end{align*}
    If $g<n \leq \frac32 g$, then $2g -n + a \geq  \frac{g}{2} +a$, then we find 
    \begin{align*}
         &\sum_{a \geq 0} 2^{a-1} (K_a+1) \rho^{K_a} \ll \rho^{g/4 }\sum_{a \geq 0} 2^{a-1} (g+a+1) \rho^{\lceil a/2\rceil} \ll \rho^{g/4 }\sum_{a \geq 0} (g+a+1) (4\rho)^{a/2} \ll g\rho^{g/4}.
    \end{align*}
    Now we observe that for large $g$, we have 
    \begin{equation*}
        g\rho^{g/4} \leq g(4\rho)^{g/4}   \leq g 2^{(\frac{3}{4A}+\frac12)g}q^{-g/4} \leq g 2^{5g+n}q^{-\theta(g)/2},
    \end{equation*}
    so we can absorb these errors into the unstable range. 
\end{proof}

\begin{prop}\label{prop:large_n}
    For $n > 2g$ we have
    \begin{equation*}
         \sum_{a=0}^{g-1} (-1)^{a} \sum_{k \geq 0} (-1)^k \mathrm{Tr}(\mathrm{Frob}_q \mid H_k^{\mathrm{\acute{e}t}}(\mathcal{H}_{g}^{1,0} \otimes \overline{\mathbb{F}}_q ,\mathbb{V}_{ (n-2g+a,1^a)})) = O(g 2^{5g+n}q^{-\theta(g)/2}).
    \end{equation*}
\end{prop}

\begin{proof}
    We bound the unstable range as before, which yields
    \begin{equation*}
        (2g) 2^{2g+n} \left(\sum_{a=0}^{g-1} 2^{a-1} \right)  \left( \sum_{k > \theta(g)} {2g \choose k} q^{-k/2} \right) \leq (2g) 2^{5g+n}q^{-\theta(g)/2}.
    \end{equation*}
    In this case, the stable range vanishes by Theorem~\ref{thm:vanishing_lambda1}.
\end{proof}

\begin{proof}[Proof of Theorem~\ref{thm:cohomological_formula}]
    By Equation~\eqref{eq:trace_formula} we have
    \begin{equation}\label{eq.exptrtd}
         \langle \operatorname{Tr}(\Theta^n_D) \rangle = \frac{q}{q-1} \sum_{\lambda \in P(n)}  c_{\lambda,n,g}  \sum_{k \geq 0}  (-1)^k \mathrm{Tr}(\mathrm{Frob}_q \mid H_k^{\mathrm{\acute{e}t}}(\mathcal{H}_g^{1,0} \otimes \overline{\mathbb{F}}_q, \mathbb{V}_\lambda)).
    \end{equation}
    Observe that if $n$ is odd, then every $\lambda \in P(n)$ has odd size. By Proposition~\ref{prop:vanishing_odd}, we conclude that $\langle \operatorname{Tr}(\Theta^n_D) \rangle = 0$ for $2 \nmid n$. For $n \leq g$ and $2 \mid n$, it follows from Lemma~\ref{lem:trivial_rep_contribution} and Proposition~\ref{prop:small_n_contribution} that
    \begin{align*}
         \langle \operatorname{Tr}(\Theta^n_D) \rangle  =
        -1 + O\left(2^{n/2}q^{-n/8}+g2^{6g+n}q^{-\theta(g)/2}\right), \ \ n\leq g, \ \ 2 \mid n, \ \ q>16.
    \end{align*}
    For $g < n \leq 2g$ and $2 \mid n$, by Lemma~\ref{lem:trivial_rep_contribution}, Proposition~\ref{prop:small_n_contribution}, and Proposition~\ref{prop:correction_term}, we have
    \begin{equation*}
        \langle \operatorname{Tr}(\Theta^n_D) \rangle  = -1 - \frac{\mathbbm{1}_{n=2g}}{q-1} + O\left(g2^{5g+n}q^{-\theta(g)/2}\right), \ \ g < n \leq 2g, \ \ 2 \mid n, \ \ 
        q>2^{3/A+2}.
    \end{equation*}
    For $n > 2g$ and $2 \mid n$, by Proposition~\ref{prop:small_n_contribution} and Proposition~\ref{prop:large_n}, we have
    \begin{equation*}
        \langle \operatorname{Tr}(\Theta^n_D) \rangle  = O\left(g2^{5g+n}q^{-\theta(g)/2}\right), \ \ n>2g, \ \ 2 \mid n.
    \end{equation*}
By combining equation~\eqref{eq.exptrtd} with the last three equations and recognizing the first term as coming from random matrix theory (cf. \cite{DiaconisShahshahani1994}),
    \begin{equation}\label{eq:trace_powers_USP2g}
        \int_{\mathrm{USp_{2g}}(\mathbb{C})} \operatorname{Tr}(U^n)dU = \begin{cases}
            2g & \text{ if } n= 0, \\
            -\mathbbm{1}_{2 \mid n} & \text{ if } 0 < n \leq 2g, \\
            0 & \text{ if } n > 2g,
        \end{cases}   
    \end{equation}
    we conclude the proof. 
\end{proof}

\section{Proof of corollaries}
We give standard arguments to derive Corollary~\ref {cor:one_level_density} and Corollary~\ref{cor:non_vanishing} from Theorem~\ref{thm:cohomological_formula}. 

\begin{proof}[Proof of Corollary~\ref{cor:one_level_density}]
    Recall that  we defined $Z_f$ in Equation~\eqref{eq:def_ZfU}. We have its Fourier expansion  given by
    \begin{equation}\label{eq:ZfU_fourier}
        Z_f(U) = \int_{-\infty}^\infty f(x)dx +\frac{1}{N}\sum_{n \neq 0} \widehat{f}\left(\frac{n}{N}\right)\operatorname{Tr}U^n
    \end{equation}
    for $U$ a unitary $N \times N$ matrix. Averaging Equation~\eqref{eq:ZfU_fourier} over $D \in \mathcal{H}_{2g+1}$, taking $N=2g$, using that $f$ is even, and applying Theorem~\ref{thm:cohomological_formula} yields
    \begin{align}\label{eq:one_level_density_expansion}
    \begin{split}
         \langle Z_f(\Theta_D) \rangle &= \widehat{f}(0) - \frac{1}{g} \sum_{1 \leq m \leq g} \widehat{f}\left(\frac{m}{g}\right) - \frac{\widehat{f}(1)}{g(q-1)} + \frac{1}{g}\sum_{m > 0} \widehat{f}\left(\frac{m}{g}\right) c_m \\
        & + O \left(\frac{1}{g}\sum_{m > 0} \widehat{f}\left(\frac{m}{g}\right)g2^{6g+2m}q^{-\lfloor 2Ag+A+B\rfloor /2}\right),
    \end{split}
    \end{align}    
    where $c_m\ll 2^mq^{-m/4}$ denotes the first term in the error term of Theorem~\ref{thm:cohomological_formula}. 

    Using the average value of traces of powers in $\mathrm{USp}_{2g}(\mathbb{C})$ from Equation~\eqref{eq:trace_powers_USP2g},   
    for even $f$ we have
    \begin{equation*}
        \int_{\mathrm{USp}_{2g}(\mathbb{C})} Z_f(U)dU = \widehat{f}(0) - \frac{1}{g} \sum_{1 \leq m \leq g} \widehat{f}\left(\frac{m}{g}\right).
    \end{equation*}
    For $q > 16$ we have $c_m \ll r^m$ for some absolute $|r|<1$, and since $\widehat{f}$ is smooth, for any $J \geq 1$
    \begin{align}
    \begin{split}\label{eq:correction_at_zero}
        \sum_{m > 0} \widehat{f}\left(\frac{m}{g}\right)c_m &= \sum_{j = 0}^{J-1} \frac{\widehat{f}^{(2j)}(0)}{(2j)!g^{2j}} \sum_{m \geq 1} m^{2j}c_m + O\left(\frac{||\widehat{f}^{(2J)}{}||_\infty}{(2J)!g^{2J}}\sum_{m \geq 1}m^{2J} r^m\right)\\
        & = \sum_{j = 0}^{J-1} C_j(q)\widehat{f}^{(2j)}(0) g^{-2j} + O_{f,J}(g^{-2J}).   
    \end{split}
    \end{align}
    Note that one can use ~\cite{rudnick2008traces} to show that
    \begin{equation*}
       c_m=q^{-m}\sum_{\substack{\deg P \mid m  \\ P \text{ irred}}} \frac{\deg P}{|P|+1},
    \end{equation*}
    and compute the $C_j(q)$ explicitly. Lastly, using $\mathrm{supp}(\widehat{f}) \subseteq (-v,v)$ we have
    \begin{equation}\label{eq:error_term}
        \frac{1}{g}\sum_{m > 0} \widehat{f}\left(\frac{m}{g}\right)g2^{6g+2m}q^{-\lfloor 2Ag+A+B\rfloor /2} \leq  vg||\widehat{f}||_\infty 2^{(6+2v)g}q^{-\lfloor 2Ag+A+B\rfloor /2}.
    \end{equation}
    Combining Equation~\eqref{eq:one_level_density_expansion}, Equation~\eqref{eq:trace_powers_USP2g}, Equation~\eqref{eq:correction_at_zero}, and Equation~\eqref{eq:error_term}, we conclude 
    \begin{align*}
        \langle Z_f(\Theta_D) \rangle &= \int_{\mathrm{USp_{2g}(\mathbb{C})}} Z_f(U)dU - \frac{\widehat{f}(1)}{g(q-1)} + \sum_{j= 0}^{J-1}C_j(q) \widehat{f}^{(2j)}(0)g^{-2j-1} \\
        &+O_{J,f}\left(g^{-2J}\right) + O_f\left(vg2^{(6+2v)g}q^{-\lfloor 2Ag+A+B\rfloor /2}\right).
    \end{align*}
\end{proof}

\begin{proof}[Proof of Corollary~\ref{cor:non_vanishing}]
    Let 
    \begin{equation*}
        p_m(g) = \frac{\#\{D \in \mathcal{H}_{2g+1} : \operatorname{ord}_{s=1/2} L(s,\chi_D) = m\}}{\#\mathcal{H}_{2g+1}}.
    \end{equation*}
    Then for any $f \in \mathcal{S}(\mathbb{R})$ with $\text{supp } \widehat{f} \subseteq (-v,v)$ with $f \geq 0$ and $f(0) = 1$,
    \begin{equation*}
        \langle Z_f(\Theta_D)\rangle = \frac{1}{\#\mathcal{H}_{2g+1}}\sum_{D \in \mathcal{H}_{2g+1}} \sum_{j=1}^{2g} F(\theta_{D,j}) \geq \sum_{m\geq 0} mp_m(g) \geq 2(1-p_0(g)),
    \end{equation*}
    where $\theta_{D,j}$ runs over the eigenvalues of $\Theta_D$.
    For the last inequality, we used $\sum_{m \geq 0} p_m(g)  = 1$, and since this family has root number $1$, we have $p_m(g) = 0$ for $m$ odd. It follows that
    \begin{equation*}
        p_0(g) \geq 1- \frac12\langle Z_f(\Theta_D) \rangle.
    \end{equation*}
    Using Corollary~\ref{cor:one_level_density}, for $q > \max \{2^{\frac{3}{A}+2}, 2^{\frac{6+2v}{A}}\}$ we have
    \begin{equation*}
        p_0(g) \geq 1 - \frac12\int_{\mathrm{USp_{2g}}(\mathbb{C})}Z_f(U)dU + o(1).
    \end{equation*}
    Therefore
    \begin{equation*}
        \liminf_{g \to \infty} p_0(g) \geq 1 - \frac{1}{2}\int_\mathbb{R}\left(1-\frac{\sin(2\pi x)}{2 \pi x}\right)f(x)dx,
    \end{equation*}
    and choosing a sequence of non-negative Schwarz functions $f_n$ with $\operatorname{supp} \widehat{f_n} \subseteq (-v,v)$ and $f_n(0)=1$ converging to
    \begin{equation*}
        f(x) = \left(\frac{\sin (\pi v x)}{\pi v x}\right)^2,
    \end{equation*}
    yields
    \begin{equation*}
       \liminf_{g \to \infty} p_0(g) \geq 1 - \frac{1}{4 v^2}.
    \end{equation*}
    Since we can take $v \to \infty$ as $q \to \infty$, we conclude
    \begin{equation*}
         \lim_{q \to \infty}\liminf_{g \to \infty} p_0(g) = 1.
    \end{equation*}
\end{proof}

\section{Proof of Theorem~\ref{thm:cubic_Kummer}}
Most of the techniques and ideas here were developed in  \cite{DFL}. 
For any $a \in \mathbb{F}_q((1/T))$, we consider the following exponential function introduced in \cite{hayes}:
\begin{equation*}
e_q(a) = e^{\frac{2 \pi i \tr_{\F_q/\F_p}(a_1) }{p}},
\label{exp-ff}
\end{equation*} 
where $a_1$ is the coefficient of $1/T$ in the Laurent expansion of $a$. 
We will repeatedly make use of generalized Gauss sums studied in \cite[Section~2C]{DFL}:
\begin{equation*}
G_q(V,f) = \sum_{u \pmod f} \chi_f(u) e_q\left( \frac{uV}{f} \right ).
\end{equation*}
Notice that, for $a \in \F_q^*$, we have \begin{equation}\label{eq:charF}\chi_c(a)=a^{\frac{q-1}{3}\deg(c)}.\end{equation}

We first consider the case $3 \nmid d$. By Equation~\eqref{eq:charF}, $\chi_c$ is an odd character. By \cite[Corollary 2.3]{DFL}, we have
\[\omega(\chi_c)=\frac{q^{-d/2}}{\epsilon(\chi_c)}G_q(1,c),\]
where
\[\epsilon(\chi_c)=q^{-1/2}\sum_{a\in \F_q^*}\chi_c(a)e^{2\pi i \tr_{\F_q/\F_p}(a)/p}.\]
Notice that $\epsilon(\chi_c)$ depends only on $\deg(c)=d$. We will denote this value $\epsilon(d)$.
We have
\begin{equation*}\sum_{\chi_c\in \mathcal{F}_3(d)}  \omega(\chi_c)\overline{\chi_c}(P) = \frac{q^{-d/2}}{\epsilon(d)} \sum_{c\in \mathcal{M}_d} G_q(1,c)\overline{\chi_c}(P) =\frac{q^{-d/2}}{\epsilon(d)} \sum_{\substack{c\in \mathcal{M}_d\\(c,P)=1}} G_q(P,c),
\end{equation*}
where we have applied the fact that $G_q(1,c)=0$ when $c$ is not squarefree (see \cite[Lemma 2.12]{DFL}).
By \cite[Proposition 3.1]{DFL}, we have, for $2/3<\sigma< 4/3$,
\begin{align*}\sum_{\substack{c\in \mathcal{M}_d\\(c,P)=1}} G_q(P,c)=&
\frac{ q^{\frac{4d}{3}-\frac{2n}{3} - \frac{4}{3} [d+n]_3} }{ \zeta_q(2)} \overline{G_q(1,P)} \rho(1, [d+n]_3)\left ( 1+\frac{1}{|P|}\right )^{-1} \\ \nonumber
&+ O \left (q^{\frac{d}{3}+\frac{n}{6}+\varepsilon d}\right )
+ O\left(q^{\sigma d+\frac{n}{2}(\frac{3}{2}-\sigma)}\right),
\end{align*}
where $[m]_3$ denotes the residue of $m$ modulo 3 such that $0\leq m \leq 2$, and
\[\rho(1,0)=1, \qquad \rho(1,1)=\tau(\chi_3)q, \qquad \rho(1,2)=0,\]
in which
\[\tau(\chi_3)=\sum_{a\in \F_q^*}\chi_3(a)e^{2\pi i \tr_{\F_q/\F_p}(a)/p}\]
where $\chi_3(a)=a^\frac{q-1}{3}$ for certain fixed choice of third root of unity in $\F_q^*.$
By \cite[Lemma 2.12]{DFL}, we have that
\[G_q(1,P)=\epsilon(\chi_P)\omega(\chi_P)|P|^{1/2}=\epsilon(n)\omega(\chi_P)|P|^{1/2}.\]
Putting all of this together with  $d=2n$ and $\sigma=2/3+\varepsilon$, we obtain
\begin{equation}\label{eq:3nmidd}
\frac{1}{\# \mathcal{F}_3(d)} \sum_{\chi_c\in \mathcal{F}_3(d)} \omega(\chi_c)\overline{\chi_c}(P)|P|^{1/2}
 = \overline{\omega(\chi_P)}\left ( 1+\frac{1}{|P|}\right )^{-1} + O \left (q^{-\frac{3d}{8}+\varepsilon d}\right ).
\end{equation}

Now consider the case $3 \mid d$. By Equation~\eqref{eq:charF}, $\chi_c$ is an even character. By \cite[Corollary 2.3]{DFL}, we have
\[\omega(\chi_c)=q^{-d/2} G_q(1,c).\]
This gives, as before,
\begin{equation*}
\sum_{\chi_c\in \mathcal{F}_3(d)}  \omega(\chi_c)\overline{\chi_c}(P) = q^{-d/2}  \sum_{c\in \mathcal{M}_d} G_q(1,c)\overline{\chi_c}(P)=q^{-d/2}  \sum_{\substack{c\in \mathcal{M}_d\\(c,P)=1}} G_q(P,c).
\end{equation*}
Applying \cite[Proposition 3.1]{DFL} as before, setting $d=2n$ and $\sigma=2/3+\varepsilon$, we have,
\begin{equation}\label{eq:3midd}
\frac{1}{\# \mathcal{F}_3(d)} \sum_{\chi_c\in \mathcal{F}_3(d)} \omega(\chi_c)\overline{\chi_c}(P)|P|^{1/2}
 = \overline{\omega(\chi_P)}\left ( 1+\frac{1}{|P|}\right )^{-1} + O \left (q^{-\frac{d}{6}+\varepsilon d}\right).
\end{equation}
Combining Equations~\eqref{eq:3nmidd} and~\eqref{eq:3midd} gives the result.

\bibliographystyle{alpha}  
\bibliography{ref}  

@book{katzsarnak1999,
    AUTHOR = {Katz, Nicholas M. and Sarnak, Peter},
     TITLE = {Random matrices, {F}robenius eigenvalues, and monodromy},
    SERIES = {American Mathematical Society Colloquium Publications},
    VOLUME = {45},
 PUBLISHER = {American Mathematical Society, Providence, RI},
      YEAR = {1999},
     PAGES = {xii+419},
      ISBN = {0-8218-1017-0},
   MRCLASS = {11G25 (11M06 11Y35 14D05 14G10 60F99 82B44)},
  MRNUMBER = {1659828},
MRREVIEWER = {Philippe\ G.\ Michel},
       DOI = {10.1090/coll/045},
       URL = {https://doi.org/10.1090/coll/045},
}

@book{macdonald1995symmetric,
  title={Symmetric Functions and Hall Polynomials},
  author={Macdonald, Ian Grant},
  year={1995},
  publisher={Oxford University Press},
  edition={2nd},
  isbn={978-0198504504},
  series={Oxford Mathematical Monographs}
}

@article {rudnick2008traces,
    AUTHOR = {Rudnick, Ze\'{e}v},
     TITLE = {Traces of high powers of the {F}robenius class in the
              hyperelliptic ensemble},
   JOURNAL = {Acta Arith.},
  FJOURNAL = {Acta Arithmetica},
    VOLUME = {143},
      YEAR = {2010},
    NUMBER = {1},
     PAGES = {81--99},
      ISSN = {0065-1036,1730-6264},
   MRCLASS = {11M50 (11G20)},
  MRNUMBER = {2640060},
MRREVIEWER = {Steven\ Joel\ Miller},
       DOI = {10.4064/aa143-1-5},
       URL = {https://doi.org/10.4064/aa143-1-5},
}

@article{wang2024noteszetaratiostabilization,
      title={Notes on zeta ratio stabilization}, 
      author={Victor Y. Wang},
      year={2024},
      journal={arXiv:2402.01214},
      archivePrefix={arXiv},
      primaryClass={math.NT},
      url={https://arxiv.org/abs/2402.01214}, 
}

@article{bergström2024hyperellipticcurvesscanningmap,
      title={Hyperelliptic curves, the scanning map, and moments of families of quadratic {$L$}-functions}, 
      author={Jonas Bergström and Adrian Diaconu and Dan Petersen and Craig Westerland},
      year={2023},
      journal={arXiv:2302.07664},
      archivePrefix={arXiv},
      primaryClass={math.NT},
      url={https://arxiv.org/abs/2302.07664}, 
}

@article {AbramovichCortiVistoli2003,
    AUTHOR = {Abramovich, Dan and Corti, Alessio and Vistoli, Angelo},
     TITLE = {Twisted bundles and admissible covers},
      NOTE = {Special issue in honor of Steven L. Kleiman},
   JOURNAL = {Comm. Algebra},
  FJOURNAL = {Communications in Algebra},
    VOLUME = {31},
      YEAR = {2003},
    NUMBER = {8},
     PAGES = {3547--3618},
      ISSN = {0092-7872,1532-4125},
   MRCLASS = {14H10 (14A20 14H30)},
  MRNUMBER = {2007376},
MRREVIEWER = {Andrew\ Kresch},
       DOI = {10.1081/AGB-120022434},
       URL = {https://doi.org/10.1081/AGB-120022434},
}

@article{miller2025uniformtwistedhomologicalstability,
      title={Uniform twisted homological stability}, 
      author={Jeremy Miller and Peter Patzt and Dan Petersen and Oscar Randal-Williams},
      year={2024},
      journal={arXiv:2402.00354},
      archivePrefix={arXiv},
      primaryClass={math.AT},
      url={https://arxiv.org/abs/2402.00354}, 
}

@article{Andrade_2014,
    AUTHOR = {Andrade, Julio C. and Keating, Jonathan P.},
     TITLE = {Conjectures for the integral moments and ratios of
              {$L$}-functions over function fields},
   JOURNAL = {J. Number Theory},
  FJOURNAL = {Journal of Number Theory},
    VOLUME = {142},
      YEAR = {2014},
     PAGES = {102--148},
      ISSN = {0022-314X,1096-1658},
   MRCLASS = {11G20 (11M50 14G10)},
  MRNUMBER = {3208396},
MRREVIEWER = {Steven\ Joel\ Miller},
       DOI = {10.1016/j.jnt.2014.02.019},
       URL = {https://doi.org/10.1016/j.jnt.2014.02.019},
}

@misc{sarnak2023murmurations,
  author       = {Sarnak, Peter},
  title        = {Letter to {D}rew {S}utherland and {N}ina {Z}ubrilina on {M}urmurations and {R}oot {N}umbers},
  year         = {2023},
  month        = aug,
  howpublished = {Handwritten letter, Institute for Advanced Study Publications},
  url          = {https://publications.ias.edu/sarnak/paper/2726},
  note         = {Listed as ``Letter to Sutherland and Zubrilina''},
}

@article {LOP,
    AUTHOR = {Lee, Kyu-Hwan and Oliver, Thomas and Pozdnyakov, Alexey},
     TITLE = {Murmurations of {D}irichlet characters},
   JOURNAL = {Int. Math. Res. Not. IMRN},
  FJOURNAL = {International Mathematics Research Notices. IMRN},
      YEAR = {2025},
    NUMBER = {1},
     PAGES = {Paper No. rnae277, 28},
      ISSN = {1073-7928,1687-0247},
   MRCLASS = {11R42},
  MRNUMBER = {4847277},
MRREVIEWER = {Soun-Hi\ Kwon},
       DOI = {10.1093/imrn/rnae277},
       URL = {https://doi.org/10.1093/imrn/rnae277},
}

@article {HLOP,
    AUTHOR = {He, Yang-Hui and Lee, Kyu-Hwan and Oliver, Thomas and
              Pozdnyakov, Alexey},
     TITLE = {Murmurations of elliptic curves},
   JOURNAL = {Exp. Math.},
  FJOURNAL = {Experimental Mathematics},
    VOLUME = {34},
      YEAR = {2025},
    NUMBER = {3},
     PAGES = {528--540},
      ISSN = {1058-6458,1944-950X},
   MRCLASS = {11G05 (11F30 11Y70 62J12)},
  MRNUMBER = {4960171},
       DOI = {10.1080/10586458.2024.2382361},
       URL = {https://doi.org/10.1080/10586458.2024.2382361},
}

@article {Zubrilina,
    AUTHOR = {Zubrilina, Nina},
     TITLE = {Murmurations},
   JOURNAL = {Invent. Math.},
  FJOURNAL = {Inventiones Mathematicae},
    VOLUME = {241},
      YEAR = {2025},
    NUMBER = {3},
     PAGES = {627--680},
      ISSN = {0020-9910,1432-1297},
   MRCLASS = {11F30 (11F11)},
  MRNUMBER = {4946245},
       DOI = {10.1007/s00222-025-01347-8},
       URL = {https://doi.org/10.1007/s00222-025-01347-8},
}

@book {Rosen,
    AUTHOR = {Rosen, Michael},
     TITLE = {Number theory in function fields},
    SERIES = {Graduate Texts in Mathematics},
    VOLUME = {210},
 PUBLISHER = {Springer-Verlag, New York},
      YEAR = {2002},
     PAGES = {xii+358},
      ISBN = {0-387-95335-3},
   MRCLASS = {11R58 (11R60 11T55)},
  MRNUMBER = {1876657},
MRREVIEWER = {Ernst-Ulrich\ Gekeler},
       DOI = {10.1007/978-1-4757-6046-0},
       URL = {https://doi.org/10.1007/978-1-4757-6046-0},
}

@article {DFL,
    AUTHOR = {David, Chantal and Florea, Alexandra and Lal\'{\i}n, Matilde},
     TITLE = {The mean values of cubic {$L$}-functions over function fields},
   JOURNAL = {Algebra Number Theory},
  FJOURNAL = {Algebra \& Number Theory},
    VOLUME = {16},
      YEAR = {2022},
    NUMBER = {5},
     PAGES = {1259--1326},
      ISSN = {1937-0652,1944-7833},
   MRCLASS = {11R59 (11M06 11M38 11R16 11R58)},
  MRNUMBER = {4471042},
MRREVIEWER = {Ofir\ Gorodetsky},
       DOI = {10.2140/ant.2022.16.1259},
       URL = {https://doi.org/10.2140/ant.2022.16.1259},
}

@article {DFL2,
    AUTHOR = {David, Chantal and Florea, Alexandra and Lalin, Matilde},
     TITLE = {Nonvanishing for cubic {$L$}-functions},
   JOURNAL = {Forum Math. Sigma},
  FJOURNAL = {Forum of Mathematics. Sigma},
    VOLUME = {9},
      YEAR = {2021},
     PAGES = {Paper No. e69, 58},
      ISSN = {2050-5094},
   MRCLASS = {11R59 (11M06 11M38 11R16 11R58)},
  MRNUMBER = {4323990},
MRREVIEWER = {Jos\'{e}\ Alejandro\ Lara Rodr\'{\i}guez},
       DOI = {10.1017/fms.2021.62},
       URL = {https://doi.org/10.1017/fms.2021.62},
}

@article {Bui-Florea,
    AUTHOR = {Bui, Hung M. and Florea, Alexandra},
     TITLE = {Zeros of quadratic {D}irichlet {$L$}-functions in the
              hyperelliptic ensemble},
   JOURNAL = {Trans. Amer. Math. Soc.},
  FJOURNAL = {Transactions of the American Mathematical Society},
    VOLUME = {370},
      YEAR = {2018},
    NUMBER = {11},
     PAGES = {8013--8045},
      ISSN = {0002-9947,1088-6850},
   MRCLASS = {11M38 (11M06 11M50)},
  MRNUMBER = {3852456},
MRREVIEWER = {Kohji\ Matsumoto},
       DOI = {10.1090/tran/7317},
       URL = {https://doi.org/10.1090/tran/7317},
}

@article {EllenbergLiShusterman,
    AUTHOR = {Ellenberg, Jordan S. and Li, Wanlin and Shusterman, Mark},
     TITLE = {Nonvanishing of hyperelliptic zeta functions over finite
              fields},
   JOURNAL = {Algebra Number Theory},
  FJOURNAL = {Algebra \& Number Theory},
    VOLUME = {14},
      YEAR = {2020},
    NUMBER = {7},
     PAGES = {1895--1909},
      ISSN = {1937-0652,1944-7833},
   MRCLASS = {11M38},
  MRNUMBER = {4150253},
MRREVIEWER = {Adam\ Morgan},
       DOI = {10.2140/ant.2020.14.1895},
       URL = {https://doi.org/10.2140/ant.2020.14.1895},
}

@article {LiHyperellipticVanishing,
    AUTHOR = {Li, Wanlin},
     TITLE = {Vanishing of hyperelliptic {L}-functions at the central point},
   JOURNAL = {J. Number Theory},
  FJOURNAL = {Journal of Number Theory},
    VOLUME = {191},
      YEAR = {2018},
     PAGES = {85--103},
      ISSN = {0022-314X,1096-1658},
   MRCLASS = {11G40},
  MRNUMBER = {3825462},
MRREVIEWER = {Steven\ Joel\ Miller},
       DOI = {10.1016/j.jnt.2018.03.018},
       URL = {https://doi.org/10.1016/j.jnt.2018.03.018},
}

@article {DonepudiLi,
    AUTHOR = {Donepudi, Ravi and Li, Wanlin},
     TITLE = {Vanishing of {D}irichlet {$L$}-functions at the central point
              over function fields},
   JOURNAL = {Rocky Mountain J. Math.},
  FJOURNAL = {The Rocky Mountain Journal of Mathematics},
    VOLUME = {51},
      YEAR = {2021},
    NUMBER = {5},
     PAGES = {1615--1628},
      ISSN = {0035-7596,1945-3795},
   MRCLASS = {11M38 (11R59 11S40 11T22 14G10)},
  MRNUMBER = {4382986},
       DOI = {10.1216/rmj.2021.51.1615},
       URL = {https://doi.org/10.1216/rmj.2021.51.1615},
}

@article{DFL3,
      title={Nonvanishing of {$L$}--functions associated to fixed order characters over function fields}, 
      author={Chantal David and Alexandra Florea and Matilde Lalin},
      year={2025},
      journal={arXiv:2506.07815},
      archivePrefix={arXiv},
      primaryClass={math.NT},
      url={https://arxiv.org/abs/2506.07815}, 
}

@article{DiaconisShahshahani1994,
  author  = {Diaconis, Persi and Shahshahani, Mehrdad},
  title   = {On the Eigenvalues of Random Matrices},
  journal = {Journal of Applied Probability},
  volume  = {31},
  number  = {A},
  series  = {Studies in Applied Probability},
  pages   = {49--62},
  year    = {1994},
  doi     = {10.1017/S0021900200106989}
}

@article{DavidGuloglu2022OneLevelDensity,
  author        = {David, Chantal and G{\"{u}}lo\u{g}lu, Ahmet M.},
  title         = {One-Level Density and Non-Vanishing for Cubic {$L$}-Functions over the {E}isenstein Field},
  journal       = {Int. Math. Res. Not. IMRN},
  year          = {2022},
  volume        = {2022},
  number        = {23},
  pages         = {18833--18873},
  doi           = {10.1093/imrn/rnab240},
  eprint        = {2102.02469},
  archiveprefix = {arXiv},
  primaryclass  = {math.NT},
  url           = {https://doi.org/10.1093/imrn/rnab240}
}

@article{Guloglu2025NonVanishing,
  author        = {G{\"u}lo\u{g}lu, Ahmet M.},
  title         = {Non-Vanishing of Cubic {D}irichlet {$L$}-Functions over the {E}isenstein Field},
  journal       = {Proceedings of the American Mathematical Society},
  year          = {2025},
  volume        = {153},
  number        = {5},
  pages         = {1947--1961},
  doi           = {10.1090/proc/17155},
  eprint        = {2306.09474},
  archiveprefix = {arXiv},
  primaryclass  = {math.NT},
  url           = {https://doi.org/10.1090/proc/17155}
}

@article {hayes,
    AUTHOR = {Hayes, David R.},
     TITLE = {The expression of a polynomial as a sum of three irreducibles},
   JOURNAL = {Acta Arith.},
  FJOURNAL = {Polska Akademia Nauk. Instytut Matematyczny. Acta Arithmetica},
    VOLUME = {11},
      YEAR = {1966},
     PAGES = {461--488},
      ISSN = {0065-1036},
   MRCLASS = {12.25},
  MRNUMBER = {0201422},
MRREVIEWER = {L. Carlitz},
       DOI = {10.4064/aa-11-4-461-488},
       URL = {https://doi.org/10.4064/aa-11-4-461-488},
}

\end{document}